\documentclass{article}
\usepackage{amssymb}
\usepackage{amsfonts}
\usepackage{amsmath}
\usepackage[margin=1.15in]{geometry}

\newtheorem{theorem}{Theorem}[section]

\newtheorem{corollary}[theorem]{Corollary}

\newtheorem{definition}[theorem]{Definition}
\newtheorem{example}[theorem]{Example}

\newtheorem{lemma}[theorem]{Lemma}
\newtheorem{notation}[theorem]{Notation}

\newtheorem{remark}[theorem]{Remark}

\newenvironment{proof}[1][Proof]{\noindent\textbf{#1.} }{\ \rule{0.5em}{0.5em}}
\input{tcilatex}
\begin{document}

\title{On It\^{o} differential equations in rough path theory}
\author{Terry J. Lyons\thanks{%
University of Oxford and Oxford-Man Institute, Email:
terry.lyons@maths.ox.ac.uk} \ \ \ \ \ Danyu Yang\thanks{%
University of Oxford and Oxford-Man Institute, Email: yangd@maths.ox.ac.uk}}
\maketitle

\begin{abstract}
The solution of rough differential equation, driven by the It\^{o} signature
of \ a continuous local martingale, exists uniquely a.s. when the vector
field is $Lip\left( \beta \right) $ for $\beta >1$, and coincides a.s. with
the It\^{o} signature of the solution of parallel stochastic differential
equation. Moreover, the It\^{o} solution can be recovered pathwisely by
concatenating discounted Stratonovich solutions.
\end{abstract}

\section{Introduction}

It\^{o} calculus \cite{Ito1}, \cite{Ito2} is a transformation between
martingales, and lies at the bottom of various mathematical models. However,
It\^{o} calculus is not pathwise and lacks stability. The former
disadvantage is due to the fact that It\^{o} calculus respects the
probabilistic structure, and problem will occur if one tries to solve
differential equation driven by a selected sample path. On the other hand,
as demonstrated by Wong and Zakai \cite{Wong Zakai}, the solutions of
ordinary differential equations, driven by piecewise-linear approximations
of Brownian motion, converge uniformly in probability to the Stratonovich
solution as the mesh of partitions tends to zero. As a result, the solution
of It\^{o} stochastic differential equation is not stable with respect to
perturbation of the driving signal, even when the perturbation is very
natural.

There have been sustained interests in developing pathwise It\^{o} calculus.
Bichteler \cite{Bichteler}, by using factorization of operators, proved
that, the It\^{o} integral can be defined pathwisely outside a null set
depending on the integrand function. Based on Bichteler's approach,
Karandikar \cite{Karandikar} got similar result by using random time change.
F\"{o}llmer \cite{Follmer} proposed a deterministic approach to integrate
closed one-forms. He proved that, for a semi-martingale $X$, if the
quadratic variation of $X$ converge pathwisely (along a sequence of finite
partitions), then the It\^{o} Riemann sums of $\int f\left( X\right) dX$
also converge pathwisely (along the same sequence of finite partitions).
Russo and Vallois \cite{Russo and Vallois} developed an almost pathwise
approach by regularizing integrals, but their method is not truly pathwise
because their convergence is in probability.

Rough path \cite{Terry98}, \cite{Terrysnotes}, \cite{TerryQian}, \cite{Peter
Friz} is close in spirit to F\"{o}llmer's approach \cite{Follmer}, but a far
more systematic methodology which is stable under a large class of
approximations, and applies not only to semi-martingales but\ also to much
wider classes of processes \cite{TerryStoica}, \cite{CoutinQian}, \cite%
{FrizVicoirGuassian}, \cite{HairerWeber}, \cite{TerryYang} etc.. As a
natural generalization of classical calculus, rough path is essentially
pathwise, and can provide a stable solution which is continuous with respect
to the driving signal (in rough path metric).

However, there is some innate non-geometric property of the It\^{o} integral
which impedes a direct application of rough path theory. The set of
(geometric) rough paths is defined as the closure of continuous bounded
variation paths in rough path metric. Thus, Stratonovich integral, as the
limit of piecewise-linear approximations, is very natural in rough path. On
the other hand, It\^{o} integral generally is not the limit of continuous
bounded variation paths. Indeed, for $2$-dimensional Brownian motion $%
B_{t}=B_{t}^{1}e_{1}+B_{t}^{2}e_{2}$, $t\in \left[ 0,1\right] $, there does 
\textit{not} exist a sequence of continuous bounded variation paths $%
B_{t}^{n}=B_{t}^{1,n}e_{1}+B_{t}^{2,n}e_{2}$, $t\in \left[ 0,1\right] $, $%
n\geq 1$, such that $B_{1}^{n}$ converge to $B_{1}$ (assuming $%
B_{0}^{n}=B_{0}=0$) and $\int_{0}^{1}B_{u}^{n}\otimes dB_{u}^{n}$ ($\otimes $
is the tensor product) converge to the It\^{o} integral $\int_{0}^{1}B_{u}%
\otimes dB_{u}$ in probability as $n\rightarrow \infty $. The reason is
that, for any $B^{n}$, the Riemann-Stieltjes integral satisfies, 
\begin{eqnarray*}
&&\int_{0}^{1}B_{u}^{n}\otimes dB_{u}^{n}-\frac{1}{2}\left( B_{1}^{n}\right)
^{\otimes 2} \\
&=&\int_{0}^{1}\left( B_{u}^{1,n}e_{1}+B_{u}^{2,n}e_{2}\right) \otimes
d\left( B_{u}^{1,n}e_{1}+B_{u}^{2,n}e_{2}\right) -\frac{1}{2}\left(
B^{1,n}e_{1}+B^{2,n}e_{2}\right) ^{\otimes 2} \\
&=&\frac{1}{2}\left(
\int_{0}^{1}B_{u}^{1,n}dB_{u}^{2,n}-\int_{0}^{1}B_{u}^{2,n}dB_{u}^{1,n}%
\right) \left( e_{1}\otimes e_{2}-e_{2}\otimes e_{1}\right) \text{.}
\end{eqnarray*}%
While the It\^{o} integral satisfies,%
\begin{eqnarray*}
&&\int_{0}^{1}B_{u}\otimes dB_{u}-\frac{1}{2}\left( B_{1}\right) ^{\otimes 2}
\\
&=&\frac{1}{2}\left(
\int_{0}^{1}B_{u}^{1}dB_{u}^{2}-\int_{0}^{1}B_{u}^{2}dB_{u}^{1}\right)
\left( e_{1}\otimes e_{2}-e_{2}\otimes e_{1}\right) +\frac{1}{2}\left(
e_{1}\otimes e_{1}+e_{2}\otimes e_{2}\right) \text{.}
\end{eqnarray*}%
Thus, there is a non-negligible symmetric part $2^{-1}\left( e_{1}\otimes
e_{1}+e_{2}\otimes e_{2}\right) $, which can not be approximated by any
sequence of $\left\{ B^{n}\right\} _{n}$. To put it in a more abstract way,
it is because that, the It\^{o} signature of Brownian motion, as will be
defined afterwards, is not a geometric rough process, i.e. it does not take
value in the nilpotent Lie group where the normal rough paths take value.

Since Stratonovich integral is well-defined in rough path, one possible
approach to defining It\^{o} integral in rough path is to define the It\^{o}
integral as Stratonovich integral plus a drift. Since the drift is generally
regular, the cross integrals between the continuous semi-martingale and the
drift are well-defined pathwisely as Young integrals \cite{Young L C}. This
idea is adopted in Lyons and Qian \cite{TerryQianpaper}. In Lejay and
Victoir \cite{Lejay and Victoir}, they interpret a $p$-rough path, $p\in
\lbrack 2,3)$, as the product of a weakly geometric $p$-rough path and
another continuous path with finite $2^{-1}p$-variation. Similar idea is
used in Friz and Victoir \cite{Peter Friz}, where they combine a $p$%
-geometric rough path with a continuous path with finite $q$-variation for $%
p^{-1}+q^{-1}>1$, and get very concrete estimates of solution of rough
differential equations driven by $\left( p,q\right) $-rough paths. In the
more recent \cite{Hairer and Kelly}, by using similar approach as in \cite%
{Gubinelli}, the authors embed a non-geometric rough path in a geometric
rough path, extend the result in \cite{Lejay and Victoir}. In this
manuscript, we will not try to define rough differential equation driven by $%
p$-rough paths, because there is a canonical choice when $p\in \lbrack 2,3)$%
. We interpret a $p$-rough path when $p\in \lbrack 2,3)$ as a $\left(
p,2^{-1}p\right) $-rough path, and focus on interpreting the It\^{o}
solution in rough path.

By It\^{o} solution, we mean the solution of a rough differential equation
driven by the It\^{o} signature (of a $d$-dimensional continuous local
martingale $Z$):%
\begin{equation*}
\mathcal{I}_{2}\left( Z\right) _{t}=\left( 1,Z_{t}-Z_{0},\int_{0}^{t}\left(
Z_{u}-Z_{0}\right) \otimes dZ_{u}\right) \text{, \ }t\geq 0\text{,}
\end{equation*}%
comparing with the Stratonovich solution driven by the Stratonovich signature%
\begin{equation*}
S_{2}\left( Z\right) _{t}=\left( 1,Z_{t}-Z_{0},\int_{0}^{t}\left(
Z_{u}-Z_{0}\right) \otimes \circ dZ_{u}\right) \text{, \ }t\geq 0\text{.}
\end{equation*}

We demonstrate that, when the vector field is $Lip\left( \beta \right) $ for 
$\beta >1$, the solution of rough differential equation, driven by the It%
\^{o} signature of a $d$-dimensional continuous local martingale, exists
uniquely a.s., and coincides a.s. with the It\^{o} signature of the
classical It\^{o} solution. As a consequence, It\^{o} differential equation
in rough path is a transformation between group-valued continuous local
martingales (i.e. the It\^{o} signatures), with the first level of its
solution coincides almost surely with the solution of classical stochastic
differential equation. We also get a pathwise It\^{o}'s lemma, which
decomposes the Stratonovich signature as the sum of two rough paths: one is
a group-valued continuous local martingale and the other is constructed from
continuous bounded variation paths.

Moreover, with concrete convergent result, we want to convey the viewpoint
that\ the It\^{o} solution can be recovered pathwisely by concatenating
discounted Stratonovich solutions. We demonstrate that, the It\^{o} solution
takes into consideration a (possible) noise, and when the underlying driving
signal is polluted with the noise, the expectation of the It\^{o} solution
coincides with Stratonovich solution. More specifically, the It\^{o}
solution can be obtained by the following method: for each possible
trajectory of the underlying, we choose a noise, discount the Stratonovich
solution to balance the growth caused by the presence of the noise,
concatenate the discounted Stratonovich solutions, and let the mesh of
partitions tends to zero. As we demonstrate, this process is applicable to
general vector fields, and the concatenated Stratonovich solutions converge
uniformly to the It\^{o} solution as the mesh of partitions tends to zero.

\section{Definitions and Notations}

\begin{notation}[$T^{\left( n\right) }\left( 
\mathbb{R}
^{d}\right) ,\left\Vert \cdot \right\Vert $]
\label{Notation tensor group}For integer $n\geq 1$, we denote $T^{\left(
n\right) }\left( 
\mathbb{R}
^{d}\right) :=1\oplus 
\mathbb{R}
^{d}\oplus \cdots \oplus \left( 
\mathbb{R}
^{d}\right) ^{\otimes n}$, and denote $\pi _{k}$ as the projection of $%
T^{\left( n\right) }\left( 
\mathbb{R}
^{d}\right) $ to $\left( 
\mathbb{R}
^{d}\right) ^{\otimes k}$, $k=0,1,\dots ,n$. We equip $T^{\left( n\right)
}\left( 
\mathbb{R}
^{d}\right) $ with the norm\footnote{%
It is not sub-additive, but is equivalent to another sub-additive norm (Exer
7.38 \cite{Peter Friz}).}%
\begin{equation}
\left\Vert g\right\Vert :=\sum_{k=1}^{n}\left\vert \pi _{k}\left( g\right)
\right\vert ^{\frac{1}{k}}\text{, }\forall g\in T^{\left( n\right) }\left( 
\mathbb{R}
^{d}\right) \text{.}  \label{definition of homogeneous norm}
\end{equation}%
Define product and inverse for $g,h\in T^{\left( n\right) }\left( 
\mathbb{R}
^{d}\right) $ as%
\begin{eqnarray*}
g\otimes h &:&=\left( 1,\pi _{1}\left( g\right) +\pi _{1}\left( h\right)
,\dots ,\sum_{k=0}^{n}\pi _{k}\left( g\right) \otimes \pi _{n-k}\left(
h\right) \right) \text{,} \\
g^{-1} &:&=\left( 1,-\pi _{1}\left( g\right) ,\dots ,\sum_{k_{1}+\cdots
+k_{j}=n,1\leq k_{i}\leq n}\left( -1\right) ^{j}\pi _{k_{1}}\left( g\right)
\otimes \cdots \otimes \pi _{k_{j}}\left( g\right) \right) \text{.}
\end{eqnarray*}%
Then $\left( T^{\left( n\right) }\left( 
\mathbb{R}
^{d}\right) ,\left\Vert \cdot \right\Vert \right) $ is a free nilpotent
topological group with identity $\left( 1,0,\dots ,0\right) $.
\end{notation}

\begin{definition}[$p$-variation]
Suppose $\gamma $ is a continuous path defined on $\left[ 0,T\right] $
taking value in topological group $\left( G,\left\Vert \cdot \right\Vert
\right) $. For $1\leq p<\infty $, define the $p$-variation of $\gamma $ as%
\begin{equation}
\left\Vert \gamma \right\Vert _{p-var,\left[ 0,T\right] }:=\left(
\sup_{D\subset \left[ 0,T\right] }\sum_{k,t_{k}\in D}\left\Vert \gamma
_{t_{k}}^{-1}\gamma _{t_{k+1}}\right\Vert ^{p}\right) ^{\frac{1}{p}}\text{,}
\label{Definition of p-variation}
\end{equation}%
where we take supremum over all finite partitions $D=\left\{ t_{k}\right\}
_{k=0}^{n}$ satisfying $0=t_{0}<\cdots <t_{n}=T$.

When $p=\infty $, denote $\left\Vert \gamma \right\Vert _{\infty -var,\left[
0,T\right] }:=\sup_{0\leq s\leq t\leq T}\left\Vert \gamma _{s}^{-1}\gamma
_{t}\right\Vert $.
\end{definition}

The definition of $p$-variation at $\left( \ref{Definition of p-variation}%
\right) $ applies to continuous path taking value in Banach spaces, with $%
\gamma _{t_{k}}^{-1}\gamma _{t_{k+1}}$ replaced by $\gamma _{t_{k+1}}-\gamma
_{t_{k}}$.

\begin{definition}[$p$-rough path]
Suppose $\gamma $ is a continuous path defined on $\left[ 0,T\right] $
taking value in $\left( T^{\left( 2\right) }\left( 
\mathbb{R}
^{d}\right) ,\left\Vert \cdot \right\Vert \right) $. We say $\gamma $ is a $%
p $-rough path for some $p\in \lbrack 2,3)$ if $\left\Vert \gamma
\right\Vert _{p-var,\left[ 0,T\right] }<\infty $.
\end{definition}

\begin{definition}[$p$-rough process]
Process $X$ on $\left[ 0,T\right] $ is said to be a $p$-rough process for
some $p\in \lbrack 2,3)$, if $X\left( \omega \right) $ is a $p$-rough path
for almost every $\omega $.
\end{definition}

The Theorem below follows from Thm 3.7 in \cite{Terrysnotes}.

\begin{theorem}[Lyons]
\label{Theorem of enhancement}Suppose $\gamma $ is a $p$-rough path for $%
p\in \lbrack 2,3)$. Then for any integer $n\geq 3$, there exists a unique
continuous path $\widetilde{\gamma }$ taking value in $\left( T^{\left(
n\right) }\left( 
\mathbb{R}
^{d}\right) ,\left\Vert \cdot \right\Vert \right) $ satisfying $\pi _{k}(%
\widetilde{\gamma })=\pi _{k}\left( \gamma \right) $, $k=1,2$, and $%
\left\Vert \widetilde{\gamma }\right\Vert _{p-var,\left[ 0,T\right] }<\infty 
$.
\end{theorem}

\begin{notation}[$S_{n}\left( \protect\gamma \right) $]
\label{Notation Sn(X)}Suppose $\gamma $ is a $p$-rough path for $p\in
\lbrack 2,3)$. For $n\geq 3$, denote $S_{n}\left( \gamma \right) $ as the
enhancement of $\gamma $ to the path taking value in $T^{\left( n\right)
}\left( 
\mathbb{R}
^{d}\right) $ as in Theorem \ref{Theorem of enhancement}. For $n=1,2$,
denote $S_{1}\left( \gamma \right) :=\left( 1,\pi _{1}\left( \gamma \right)
\right) $ and $S_{2}\left( \gamma \right) :=\gamma $.
\end{notation}

Based on \cite{Lejay and Victoir}, any $p$-rough path, $p\in \lbrack 2,3)$,
can be interpreted as the product of a weak geometric $p$-rough path and
another continuous path with finite $2^{-1}p$-variation. We will use this
equivalence and define solution of rough differential equations driven by $p$%
-rough paths, $p\in \lbrack 2,3)$, in the sense of $\left( p,2^{-1}p\right) $%
-rough path in \cite{Peter Friz}.

\begin{notation}
\label{Notation X1 and X2}Suppose $\gamma :\left[ 0,T\right] \rightarrow
\left( T^{\left( 2\right) }\left( 
\mathbb{R}
^{d}\right) ,\left\Vert \cdot \right\Vert \right) $ is a $p$-rough path for
some $p\in \lbrack 2,3)$. Then we denote $\gamma :=\left( \gamma ^{1},\gamma
^{2}\right) $ with $\gamma ^{1}:\left[ 0,T\right] \rightarrow \left(
T^{\left( 2\right) }\left( 
\mathbb{R}
^{d}\right) ,\left\Vert \cdot \right\Vert \right) $ and $\gamma ^{2}:\left[
0,T\right] \rightarrow \left( 
\mathbb{R}
^{d}\right) ^{\otimes 2}$ continuous paths defined as%
\begin{eqnarray*}
\gamma _{t}^{1} &:&=1+\pi _{1}\left( \gamma _{t}\right) +\text{Anti}\left(
\pi _{2}\left( \gamma _{t}\right) \right) +\frac{1}{2}\left( \pi _{1}\left(
\gamma _{t}\right) \right) ^{\otimes 2}\text{, \ }t\in \left[ 0,T\right] 
\text{,} \\
\gamma _{t}^{2} &:&=\text{Sym}\left( \pi _{2}\left( \gamma _{t}\right)
\right) -\frac{1}{2}\left( \pi _{1}\left( \gamma _{t}\right) \right)
^{\otimes 2}\text{, \ }t\in \left[ 0,T\right] \text{,}
\end{eqnarray*}%
where Anti$\left( \cdot \right) $ denotes the projection of $\left( 
\mathbb{R}
^{d}\right) ^{\otimes 2}$ to $span\!\left\{ e_{i}\otimes e_{j}-e_{j}\otimes
e_{i}\right\} _{1\leq i<j\leq d}$ and Sym$\left( \cdot \right) $ denotes the
projection of $\left( 
\mathbb{R}
^{d}\right) ^{\otimes 2}$ to $span\!\left\{ e_{i}\otimes e_{j}+e_{j}\otimes
e_{i}\right\} _{1\leq i\leq j\leq d}$.
\end{notation}

Denote $L\left( 
\mathbb{R}
^{d},%
\mathbb{R}
^{e}\right) $ as the set of linear mappings from $%
\mathbb{R}
^{d}$ to $%
\mathbb{R}
^{e}$.

\begin{definition}[$Lip\left( \protect\beta \right) $ vector field]
$f:%
\mathbb{R}
^{e}\rightarrow L\left( 
\mathbb{R}
^{d},%
\mathbb{R}
^{e}\right) $ is said to be $Lip\left( \beta \right) $ for $\beta >0$, if $f$
is $\lfloor \beta \rfloor $-times Fr\'{e}chet differentiable ($\lfloor \beta
\rfloor $ is the largest integer which is strictly less than $\beta $) and%
\begin{equation*}
\left\vert f\right\vert _{Lip\left( \beta \right) }:=\max_{k=0,\dots
,\lfloor \beta \rfloor }\left\Vert D^{k}f\right\Vert _{\infty }\vee
\left\Vert D^{\lfloor \beta \rfloor }f\right\Vert _{\left( \beta -\lfloor
\beta \rfloor \right) -H\ddot{o}l}<\infty \text{,}
\end{equation*}%
where $\left\Vert \cdot \right\Vert _{\infty }$ denotes the uniform norm and 
$\left\Vert \cdot \right\Vert _{\left( \beta -\lfloor \beta \rfloor \right)
-H\ddot{o}l}$ denotes the $\left( \beta -\lfloor \beta \rfloor \right) $-H%
\"{o}lder norm.
\end{definition}

\begin{definition}[signature]
Suppose $x:\left[ 0,T\right] \rightarrow 
\mathbb{R}
^{d}$ is a continuous bounded variation path. For integer $n\geq 1$, we
define the signature of $x$: $s_{n}\left( x\right) :\left[ 0,T\right]
\rightarrow \left( T^{\left( n\right) }\left( 
\mathbb{R}
^{d}\right) ,\left\Vert \cdot \right\Vert \right) $ as%
\begin{equation}
s_{n}\left( x\right) _{t}:=\left(
1,x_{t}-x_{0},\iint_{0<u_{1}<u_{2}<t}dx_{u_{1}}\otimes dx_{u_{2}},\dots
,\idotsint\nolimits_{0<u_{1}<\cdots <u_{n}<t}dx_{u_{1}}\otimes \cdots
\otimes dx_{u_{n}}\right) \text{.}  \label{Definition of signature}
\end{equation}
\end{definition}

Then based on Definition 12.2 in \cite{Peter Friz}, we define solution of
rough differential equation driven by $p$-rough path for $p\in \lbrack 2,3)$%
. ($C^{1-var}\left( \left[ 0,T\right] ,%
\mathbb{R}
^{d}\right) $ denotes the set of continuous bounded variation paths defined
on $\left[ 0,T\right] $ taking value in $%
\mathbb{R}
^{d}$).

\begin{definition}[solution of RDE]
\label{Definition solution of RDE}Suppose $\gamma :\left[ 0,T\right]
\rightarrow \left( T^{\left( 2\right) }\left( 
\mathbb{R}
^{d}\right) ,\left\Vert \cdot \right\Vert \right) \ $is a $p$-rough path for
some $p\in \lbrack 2,3)$ with the decomposition$\ \gamma =\left( \gamma
^{1},\gamma ^{2}\right) $ (Notation \ref{Notation X1 and X2}), and $f:%
\mathbb{R}
^{e}\rightarrow L\left( 
\mathbb{R}
^{d},%
\mathbb{R}
^{e}\right) $ is $Lip\left( \beta \right) $ for some $\beta \geq 1$. Then
continuous path $Y:\left[ 0,T\right] \rightarrow \left( T^{\left( n\right) }(%
\mathcal{%
\mathbb{R}
}^{e}),\left\Vert \cdot \right\Vert \right) $ is said to be a solution of
the rough differential equation%
\begin{equation}
dY=f\left( \pi _{1}\left( Y\right) \right) d\gamma \text{, }Y_{0}=\mathbf{%
\xi }\in T^{\left( n\right) }\left( 
\mathbb{R}
^{e}\right) \text{,}  \label{Definition of RDE}
\end{equation}%
if there exists a sequence $\left\{ \left( \gamma ^{1,m},\gamma
^{2,m}\right) \right\} _{m}$ in $C^{1-var}\left( \left[ 0,T\right] ,%
\mathbb{R}
^{d}\right) \times C^{1-var}\left( \left[ 0,T\right] ,\left( 
\mathbb{R}
^{d}\right) ^{\otimes 2}\right) \ $satisfying (with $s_{2}\left( \cdot
\right) $ defined at $\left( \ref{Definition of signature}\right) $, $%
s_{2}\left( \cdot \right) _{s,t}:=s_{2}\left( \cdot \right) _{s}^{-1}\otimes
s_{2}\left( \cdot \right) _{t}$)%
\begin{gather*}
\sup_{m}\left( \left\Vert \gamma ^{1,m}\right\Vert _{p-var,\left[ 0,T\right]
}+\left\Vert \gamma ^{2,m}\right\Vert _{\frac{p}{2}-var,\left[ 0,T\right]
}\right) <\infty \text{,} \\
\lim_{m\rightarrow \infty }\max_{k=1,2}\sup_{0\leq s\leq t\leq T}\left\vert
\pi _{k}\left( s_{2}\left( \gamma ^{1,m}\right) _{s,t}\right) -\pi
_{k}\left( \gamma _{0,t}^{1}\right) \right\vert =0\text{, }%
\lim_{m\rightarrow \infty }\sup_{0\leq s\leq t\leq T}\left\vert \left(
\gamma _{t}^{2,m}-\gamma _{s}^{2,m}\right) -\left( \gamma _{t}^{2}-\gamma
_{s}^{2}\right) \right\vert =0\text{,}
\end{gather*}%
such that the signature of the solution of the ordinary differential
equations%
\begin{equation*}
dy^{m}=f\left( y^{m}\right) d\gamma ^{1,m}+\left( Dff\right) \left(
y^{m}\right) d\gamma ^{2,m}\text{, }y_{0}^{m}=\pi _{1}\left( \mathbf{\xi }%
\right) \in 
\mathbb{R}
^{e}\text{,}
\end{equation*}%
converge to $Y$ uniformly, i.e.%
\begin{equation*}
\lim_{m\rightarrow \infty }\max_{1\leq k\leq n}\sup_{t\in \left[ 0,T\right]
}\left\vert \pi _{k}\left( \mathbf{\xi }\otimes s_{n}\left( y^{m}\right)
_{0,t}\right) -\pi _{k}\left( Y_{t}\right) \right\vert =0\text{.}
\end{equation*}
\end{definition}

\begin{definition}[$d_{p}$ metric]
\label{Definition of dp metric}Suppose $\gamma ,\widetilde{\gamma }:\left[
0,T\right] \rightarrow \left( T^{\left( 2\right) }\left( 
\mathbb{R}
^{d}\right) ,\left\Vert \cdot \right\Vert \right) $ are two $p$-rough paths
for some $p\in \lbrack 2,3)$. We define ($\gamma _{s,t}:=\gamma
_{s}^{-1}\otimes \gamma _{t}$)%
\begin{equation*}
d_{p-var,\left[ 0,T\right] }\left( \gamma ,\widetilde{\gamma }\right)
:=\max_{k=1,2}\sup_{D\subset \left[ 0,T\right] }\left( \sum_{j,t_{j}\in
D}\left\vert \pi _{k}\left( \gamma _{t_{j},t_{j+1}}\right) -\pi _{k}\left( 
\widetilde{\gamma }_{t_{j},t_{j+1}}\right) \right\vert ^{\frac{p}{k}}\right)
^{\frac{1}{p}}\text{.}
\end{equation*}
\end{definition}

Based on Thm 12.6, Thm 12.10 and Prop 8.7 in \cite{Peter Friz}, we have

\begin{theorem}[existence, uniqueness and continuity]
\label{Theorem existence and uniqueness}The solution of $\left( \ref%
{Definition of RDE}\right) $ exists when $f$ is $Lip\left( \beta \right) $
for $\beta >p-1$, and is unique when $\beta >p$. When $\beta >p$, the
solution is continuous w.r.t. the driving rough path in $d_{p}$-metric.
\end{theorem}

\begin{definition}[integration of $1$-form]
\label{Definition integration of 1-forms}Suppose $\gamma :\left[ 0,T\right]
\rightarrow \left( T^{\left( 2\right) }\left( 
\mathbb{R}
^{d}\right) ,\left\Vert \cdot \right\Vert \right) $ is a $p$-rough path for
some $p\in \lbrack 2,3)$, and $f:%
\mathbb{R}
^{d}\rightarrow L\left( 
\mathbb{R}
^{d},%
\mathbb{R}
^{e}\right) $ is $Lip\left( \beta \right) $ for some $\beta \geq 1$. Then
continuous path $Y:\left[ 0,T\right] \rightarrow \left( T^{\left( 2\right) }(%
\mathcal{%
\mathbb{R}
}^{e}),\left\Vert \cdot \right\Vert \right) $ is said to be the rough
integral of $f$ against $\gamma $ and denoted as%
\begin{equation*}
Y_{t}=\int_{0}^{t}f\left( \pi _{1}\left( \gamma _{u}\right) \right) d\gamma
_{u}\text{, \ }t\in \left[ 0,T\right] \text{,}
\end{equation*}%
if there exists a continuous path $\Gamma :$ $\left[ 0,T\right] \rightarrow
\left( T^{\left( 2\right) }(\mathcal{%
\mathbb{R}
}^{d+e}),\left\Vert \cdot \right\Vert \right) $ satisfying $\pi _{T^{\left(
2\right) }(\mathcal{%
\mathbb{R}
}^{d})}\left( \Gamma \right) =\gamma $, $\pi _{T^{\left( 2\right) }(\mathcal{%
\mathbb{R}
}^{e})}\left( \Gamma \right) =Y$, and $\Gamma $ is a solution to the rough
differential equation: 
\begin{equation*}
d\Gamma =\left( 1,f\left( \pi _{%
\mathbb{R}
^{d}}\left( \Gamma \right) \right) \right) d\gamma \text{, \ }\Gamma
_{0}=\left( 1,\left( \pi _{1}\left( \gamma _{0}\right) ,0\right) ,0\right)
\in T^{\left( 2\right) }\left( \mathcal{%
\mathbb{R}
}^{d+e}\right) \text{.}
\end{equation*}
\end{definition}

\section{It\^{o} signature and relation with It\^{o} SDE}

\begin{definition}[It\^{o} signature $\mathcal{I}_{n}\left( Z\right) $]
Suppose $Z$ is a continuous local martingale on $[0,\infty )$ taking value
in $%
\mathbb{R}
^{d}$. For integer $n\geq 1$, denote $\mathcal{I}_{n}\left( Z\right)
:[0,\infty )\rightarrow \left( T^{\left( n\right) }\left( 
\mathbb{R}
^{d}\right) ,\left\Vert \cdot \right\Vert \right) $ as the combination of
iterated It\^{o} integrals:%
\begin{equation}
\mathcal{I}_{n}\left( Z\right) _{t}:=\left(
1,Z_{t}-Z_{0},\iint_{0<u_{1}<u_{2}<t}dZ_{u_{1}}\otimes dZ_{u_{2}},\dots
,\idotsint\limits_{0<u_{1}<\cdots <u_{n}<t}dZ_{u_{1}}\otimes \cdots \otimes
dZ_{u_{n}}\right) \text{, }\forall t\in \lbrack 0,\infty )\text{.}
\label{Definition In(Z)}
\end{equation}
\end{definition}

Then we study the (pathwise and probabilistic) regularity of the It\^{o}
signature.

Function $F:\overline{%
\mathbb{R}
^{+}}\rightarrow \overline{%
\mathbb{R}
^{+}}$ is called moderate if : $\left( i\right) $ $x\mapsto F\left( x\right) 
$ is continuous and increasing; $\left( ii\right) $ $F\left( x\right) =0$ if
and only if $x=0$; $\left( iii\right) $ for some $\alpha >1$, $\sup_{x>0}%
\frac{F\left( \alpha x\right) }{x}<\infty $.

\begin{theorem}
Suppose $Z$ is a continuous local martingale taking value in $%
\mathbb{R}
^{d}$. Then $\mathcal{I}_{n}\left( Z\right) :[0,\infty )\rightarrow \left(
T^{\left( n\right) }\left( 
\mathbb{R}
^{d}\right) ,\left\Vert \cdot \right\Vert \right) $ is a group-valued
continuous local martingale which satisfies, for any $T>0$,%
\begin{equation*}
\left\Vert \mathcal{I}_{n}\left( Z\right) \right\Vert _{p-var,\left[ 0,T%
\right] }<\infty \text{, a.s., }\forall p>2\text{.}
\end{equation*}%
Moreover, for any integer $d\geq 1$, any integer $n\geq 1$, any moderate
function $F$, and any $p>2$, there exists constant $C\left( d,n,F,p\right) $
such that 
\begin{equation}
C^{-1}E\left( F\left( \left\Vert \left\langle Z\right\rangle \right\Vert
_{\infty -var,[0,\infty )}^{\frac{1}{2}}\right) \right) \leq E\left( F\left(
\left\Vert \mathcal{I}_{n}\left( Z\right) \right\Vert _{p-var,[0,\infty
)}\right) \right) \leq CE\left( F\left( \left\Vert \left\langle
Z\right\rangle \right\Vert _{\infty -var,[0,\infty )}^{\frac{1}{2}}\right)
\right) \text{,}  \label{BDG inequality for group-valued martingale}
\end{equation}%
holds for any continuous local martingale $Z$ taking value in $%
\mathbb{R}
^{d}$\ starting from $0$.
\end{theorem}

\begin{proof}
That $\mathcal{I}_{n}\left( Z\right) $ is a group-valued continuous local
martingale can be proved e.g. by taking stopping times $\widetilde{\tau _{n}}%
:=\tau _{n}\wedge \inf \left\{ t|\left\vert Z_{t}\right\vert \geq n\right\} $
with $\left\{ \tau _{n}\right\} _{n}$ the stopping times of $Z$. Based on
Lemma \ref{Lemma group-valued martingale} on p\pageref{Lemma group-valued
martingale}, we have, ($S_{n}\left( \cdot \right) $ in Notation \ref%
{Notation Sn(X)})%
\begin{equation*}
S_{n}\left( \mathcal{I}_{2}\left( Z\right) \right) _{t}=\mathcal{I}%
_{n}\left( Z\right) _{t}\text{, }\forall t\in \lbrack 0,\infty )\text{, }%
\forall n\geq 1\text{, a.s..}
\end{equation*}%
Then based on Theorem 3.7 in \cite{Terrysnotes} and Theorem 14.9 \cite{Peter
Friz}, we have, for any $T>0$ and any $p>2$,%
\begin{eqnarray*}
\left\Vert \mathcal{I}_{n}\left( Z\right) \right\Vert _{p-var,\left[ 0,T%
\right] } &\leq &C_{p,n}\left\Vert \mathcal{I}_{2}\left( Z\right)
\right\Vert _{p-var,\left[ 0,T\right] } \\
&\leq &C_{p,n}\left( \left\Vert S_{2}\left( Z\right) \right\Vert _{p-var, 
\left[ 0,T\right] }+\left\Vert \left\langle Z\right\rangle \right\Vert
_{1-var,\left[ 0,T\right] }^{\frac{1}{2}}\right) <\infty \text{, a.s..}
\end{eqnarray*}%
Based on Theorem 3.7 in \cite{Terrysnotes} and Theorem 14.12 in \cite{Peter
Friz}, we have $\left( \ref{BDG inequality for group-valued martingale}%
\right) $ holds.
\end{proof}

Then we investigate the pathwise property of the It\^{o} solution (i.e.
solution of rough differential equation driven by the It\^{o} signature $%
\mathcal{I}_{2}\left( Z\right) $ defined at $\left( \ref{Definition In(Z)}%
\right) $).

\begin{theorem}[Relation with It\^{o} SDE]
\label{Theorem coincidence with solution to SDE}Suppose $Z$ is a continuous
local martingale on $[0,\infty )$ taking value in $%
\mathbb{R}
^{d}$. Suppose $f:%
\mathbb{R}
^{e}\rightarrow L\left( 
\mathbb{R}
^{d},%
\mathbb{R}
^{e}\right) $ is $Lip\left( \beta \right) $ for $\beta >1$. Then for almost
every sample path of $Z$, the solution of the rough differential equation%
\begin{equation}
dY=f\left( \pi _{1}\left( Y\right) \right) d\mathcal{I}_{2}\left( Z\right) 
\text{, \ }Y_{0}=\mathbf{\xi }\in T^{\left( n\right) }\left( 
\mathbb{R}
^{e}\right) \text{,}  \label{RDE driven by Ito signature}
\end{equation}%
exists uniquely, and the solution $Y$ has the explicit expression: 
\begin{equation*}
Y_{t}=\mathbf{\xi }\otimes \mathcal{I}_{n}\left( y\right) _{0,t}\text{, }%
\forall t\in \lbrack 0,\infty )\text{, }\forall n\geq 1\text{,}
\end{equation*}%
with $\mathcal{I}_{n}\left( y\right) $ defined at $\left( \ref{Definition
In(Z)}\right) $ and $y\ $denotes the unique strong continuous solution of
the It\^{o} stochastic differential equation%
\begin{equation}
dy=f\left( y\right) dZ\text{, \ }y_{0}=\pi _{1}\left( \mathbf{\xi }\right)
\in 
\mathbb{R}
^{e}\text{.}  \label{SDE in theorem}
\end{equation}
\end{theorem}

The proof of Theorem \ref{Theorem coincidence with solution to SDE} starts
from page \pageref{Proof of Theorem coincidence with solution to SDE}.

Based on Thm 17.3 \cite{Peter Friz}, the authors identified a relationship
between the classical It\^{o} solution and the (first level) rough It\^{o}
solution when the vector field is $Lip\left( \beta \right) $ for $\beta >2$.
In Prop 4.3 \cite{Davie}, the author proved that, when the vector field is $%
Lip\left( \beta \right) $ for $\beta >1$, the (first level) solution of $%
\left( \ref{RDE driven by Ito signature}\right) $ driven by the It\^{o}
signature of Brownian motion exists uniquely a.s.. Based on Theroem \ref%
{Theorem coincidence with solution to SDE}, the result in \cite{Davie} is
applicable to the whole rough path solution and to all continuous local
martingales.

Based on Theorem \ref{Theorem existence and uniqueness} (on page \pageref%
{Theorem existence and uniqueness}), when the vector field $f$ in Theorem %
\ref{Theorem coincidence with solution to SDE} is $Lip\left( \beta \right) $
for $\beta >2$, the solution $Y$ in $\left( \ref{RDE driven by Ito signature}%
\right) $ is continuous w.r.t. the driving rough path in $d_{p}$-metric
(Definition \ref{Definition of dp metric}) for any $p\in \left( 2,\beta
\right) $ (see Thm 12.10 \cite{Peter Friz} for concrete estimate).

\begin{example}
\label{Example convergence of Ito signature of piecewise linear
approximation}Suppose $Z$ is a continuous local martingale on $[0,\infty )$
taking value in $%
\mathbb{R}
^{d}$. For $T>0$ and finite partition $D=\left\{ t_{k}\right\}
_{k=0}^{n}\subset \left[ 0,T\right] $, define $Z^{D}:\left[ 0,T\right]
\rightarrow 
\mathbb{R}
^{d}$ and $\left\langle Z\right\rangle ^{D}:\left[ 0,T\right] \rightarrow
\left( 
\mathbb{R}
^{d}\right) ^{\otimes 2}$ as 
\begin{eqnarray*}
Z_{0}^{D} &:&=Z_{0}\text{, \ }Z_{t}^{D}:=\text{ }\frac{t-t_{k}}{t_{k+1}-t_{k}%
}\left( Z_{t_{k+1}}-Z_{t_{k}}\right) +Z_{t_{k}}^{D}\text{, }t\in \left[
t_{k},t_{k+1}\right] \text{, }t_{k}\in D\text{,} \\
\left\langle Z\right\rangle _{0}^{D} &:&=0\text{, }\left\langle
Z\right\rangle _{t}^{D}:=\frac{t-t_{k}}{t_{k+1}-t_{k}}\left(
Z_{t_{k+1}}-Z_{t_{k}}\right) ^{\otimes 2}+\left\langle Z\right\rangle
_{t_{k}}^{D}\text{, }t\in \left[ t_{k},t_{k+1}\right] \text{, }t_{k}\in D%
\text{,}
\end{eqnarray*}%
and define $\mathcal{I}_{2}\left( Z\right) ^{D}:\left[ 0,T\right]
\rightarrow \left( T^{\left( 2\right) }\left( 
\mathbb{R}
^{d}\right) ,\left\Vert \cdot \right\Vert \right) $ as%
\begin{equation*}
\mathcal{I}_{2}\left( Z\right) _{t}^{D}:=\left(
1,Z_{t}^{D}-Z_{0}^{D},\int_{0}^{t}\left( Z_{s}^{D}-Z_{0}^{D}\right) \otimes
dZ_{s}^{D}-\frac{1}{2}\left\langle Z\right\rangle _{t}^{D}\right) \text{.}
\end{equation*}%
Then, with $d_{p}$-metric defined in Definition \ref{Definition of dp metric}%
, we have 
\begin{equation}
\lim_{\left\vert D\right\vert \rightarrow 0,D\subset \left[ 0,T\right]
}d_{p-var,\left[ 0,T\right] }\left( \mathcal{I}_{2}\left( Z\right) ^{D},%
\mathcal{I}_{2}\left( Z\right) \right) =0\text{ in prob. for any }p>2\text{,}
\label{convergence of piecewise linear approximation}
\end{equation}%
and the convergence is in $L^{q}$ if $\left\langle Z\right\rangle _{T}$ is
in $L^{\frac{q}{2}}$ for some $q\geq 1$.
\end{example}

\begin{proof}
Denote $s_{2}\left( Z^{D}\right) $ as the step-$2$ signature of $Z^{D}$
(defined at $\left( \ref{Definition of signature}\right) $), and denote $%
S_{2}\left( Z\right) $ as the step-$2$ Stratonovich signature of $Z$. Then
for any $p>2$, (with $d_{p}$-metric defined in Definition \ref{Definition of
dp metric})%
\begin{equation*}
d_{p-var,\left[ 0,T\right] }\left( \mathcal{I}_{2}\left( Z^{D}\right) ,%
\mathcal{I}_{2}\left( Z\right) \right) \leq d_{p-var,\left[ 0,T\right]
}\left( s_{2}\left( Z^{D}\right) ,S_{2}\left( Z\right) \right) +2^{-\frac{1}{%
2}}\left\Vert \left\langle Z\right\rangle ^{D}-\left\langle Z\right\rangle
\right\Vert _{\frac{p}{2}-var,\left[ 0,T\right] }^{\frac{1}{2}}\text{.}
\end{equation*}%
Based on Thm 14.16 in \cite{Peter Friz} and Lemma \ref{Lemma strong
convergence of bracket process} (on page \pageref{Lemma strong convergence
of bracket process}), we have $\left( \ref{convergence of piecewise linear
approximation}\right) $ holds.
\end{proof}

The Corollary below follows from Definition \ref{Definition integration of
1-forms} and Theorem \ref{Theorem coincidence with solution to SDE}.

\begin{corollary}[Integration of one-forms]
\label{Corollary for pathwise Ito integral}Suppose $Z$ is a continuous local
martingale on $[0,\infty )$ taking value in $%
\mathbb{R}
^{d}$, and $f:%
\mathbb{R}
^{e}\rightarrow L\left( 
\mathbb{R}
^{d},%
\mathbb{R}
^{e}\right) $ is $Lip\left( \beta \right) $ for $\beta >1$. Then for almost
every sample path of $Z$, the rough integral has the explicit expression: 
\begin{equation*}
S_{n}\left( \int_{0}^{\cdot }f\left( Z\right) d\mathcal{I}_{2}\left(
Z\right) \right) _{t}=\mathcal{I}_{n}\left( y\right) _{t}\text{, }t\in
\lbrack 0,\infty )\text{, }\forall n\geq 1\text{,}
\end{equation*}%
with $y$ denotes the classical It\^{o} integral%
\begin{equation*}
y_{t}:=\int_{0}^{t}f\left( Z_{u}\right) dZ_{u}\text{, }t\in \lbrack 0,\infty
)\text{.}
\end{equation*}
\end{corollary}

When vector field $f$ in Corollary \ref{Corollary for pathwise Ito integral}
is $Lip\left( \beta \right) $ for $\beta >1$, the rough integral $%
\int_{0}^{\cdot }f\left( Z\right) d\mathcal{I}_{2}\left( Z\right) $ is
continuous w.r.t. the driving rough path in $d_{p}$-metric (p239 \cite%
{TerryQianpaper}).

Based on Corollary \ref{Corollary for pathwise Ito integral}, we have a
pathwise It\^{o}'s lemma, which decomposes the Stratonovich signature as the
sum of two rough paths: one is a group-valued continuous local martingale
and the other is constructed from continuous bounded variation paths.

Theorem \ref{Theorem Ito's lemma in rough path} below follows from Lyons and
Qian \cite{TerryQianpaper} (p244), only that the rough integral $t\mapsto
\int_{0}^{t}Df\left( Z_{u}\right) d\mathcal{I}_{2}\left( Z\right) _{u}$ has
the explicit expression as seen in Corollary \ref{Corollary for pathwise Ito
integral}.

\begin{theorem}[It\^{o}'s lemma]
\label{Theorem Ito's lemma in rough path}Suppose $Z$ is a continuous local
martingale on $[0,\infty )$ taking value in $%
\mathbb{R}
^{d}$, and suppose $f:%
\mathbb{R}
^{d}\rightarrow 
\mathbb{R}
^{e}$ is $Lip\left( \beta \right) $ for $\beta >2$. Denote $S_{2}\left(
f\left( Z\right) \right) $ and $S_{2}\left( Z\right) $ as the step-$2$
Stratonovich signature of $f\left( Z\right) $ and $Z$. Then the rough
integral equation holds for almost every sample path of $Z$:%
\begin{equation}
S_{2}\left( f\left( Z\right) \right) _{0,t}=\int_{0}^{t}Df\left(
Z_{u}\right) dS_{2}\left( Z\right) _{u}=\int_{0}^{t}\left( Df\left(
Z_{u}\right) d\mathcal{I}_{2}\left( Z\right) _{u}+dH_{u}\right) \text{, }%
\forall t\in \lbrack 0,\infty )\text{,}  \label{Ito's lemma}
\end{equation}%
where $H:[0,\infty )\rightarrow \left( T^{\left( 2\right) }\left( 
\mathbb{R}
^{e}\right) ,\left\Vert \cdot \right\Vert \right) $ is defined as%
\begin{gather*}
H_{t}:=\left( 1,x_{t}^{1},\int_{0}^{t}x_{u}^{1}\otimes
dx_{u}^{1}+x_{t}^{2}\right) \text{, \ }t\in \lbrack 0,\infty )\text{,} \\
\text{with }x_{t}^{1}:=\frac{1}{2}\int_{0}^{t}\left( D^{2}f\right) \left(
Z_{u}\right) d\left\langle Z\right\rangle _{u}\text{ and }x_{t}^{2}:=\frac{1%
}{2}\int_{0}^{t}\left( Df\right) \left( Z_{u}\right) \otimes \left(
Df\right) \left( Z_{u}\right) d\left\langle Z\right\rangle _{u}\text{,}
\end{gather*}%
and $\int_{0}^{\cdot }\left( Df\left( Z_{u}\right) d\mathcal{I}_{2}\left(
Z\right) _{u}+dH_{u}\right) :[0,\infty )\rightarrow \left( T^{\left(
2\right) }\left( 
\mathbb{R}
^{e}\right) ,\left\Vert \cdot \right\Vert \right) $ is defined as 
\begin{eqnarray*}
\pi _{1}\left( \int_{0}^{t}\left( Df\left( Z_{u}\right) d\mathcal{I}%
_{2}\left( Z\right) _{u}+dH_{u}\right) \right) &:&=\pi _{1}\left(
\int_{0}^{t}Df\left( Z_{u}\right) d\mathcal{I}_{2}\left( Z\right)
_{u}\right) +\pi _{1}\left( H_{t}\right) \text{,} \\
\pi _{2}\left( \int_{0}^{t}\left( Df\left( Z_{u}\right) d\mathcal{I}%
_{2}\left( Z\right) _{u}+dH_{u}\right) \right) &:&=\pi _{2}\left(
\int_{0}^{t}Df\left( Z_{u}\right) d\mathcal{I}_{2}\left( Z\right)
_{u}\right) +\pi _{2}\left( H_{t}\right) \\
&&+\int_{0}^{t}\pi _{1}\!\left( \int_{0}^{u}Df\left( Z\right) d\mathcal{I}%
_{2}\left( Z\right) \right) \otimes d\,\pi _{1}\!\left( H_{u}\right) \\
&&+\int_{0}^{t}\pi _{1}\!\left( H_{u}\right) \otimes d\,\pi _{1}\!\left(
\int_{0}^{u}Df\left( Z\right) d\mathcal{I}_{2}\left( Z\right) \right) \text{,%
}
\end{eqnarray*}%
where the cross integrals between $\pi _{1}\!\left( \int Df\left( Z\right) d%
\mathcal{I}_{2}\left( Z\right) \right) $ and $\pi _{1}\!\left( H\right) $
are defined as Young integrals.
\end{theorem}

\section{Averaging Stratonovich solutions}

As mentioned in the introduction, we want to recover the It\^{o} solution by
concatenating a mean of Stratonovich solutions. The idea is simple, but the
concrete formulation needs some care. Here we try to give a sensible
explanation of our formulation.

Suppose $\gamma :\left[ 0,T\right] \rightarrow \left( T^{\left( 2\right)
}\left( 
\mathbb{R}
^{d}\right) ,\left\Vert \cdot \right\Vert \right) $ is a fixed $p$-rough
path for some $p\in \lbrack 2,3)$, and $M$ a continuous martingale taking
value in $%
\mathbb{R}
^{d}$ with $S_{2}\left( M\right) $ denotes the step-$2$ Stratonovich
signature of $M$. Further assume that 
\begin{equation*}
\left( \gamma +S_{2}\left( M\right) \right) _{t}:=\left( 1,\pi _{1}\left(
\gamma _{t}\right) +M_{t},\pi _{2}\left( \gamma _{t}\right) +\int_{0}^{t}\pi
_{1}\left( \gamma _{u}\right) \otimes \circ dM_{u}+\int_{0}^{t}M_{u}\otimes
\circ d\pi _{1}\left( \gamma _{u}\right) +\int_{0}^{t}M_{u}\otimes \circ
dM_{u}\right)
\end{equation*}%
is a $p$-rough process for some $p\in \left( 2,3\right) $. We want to get
the mathematical expression of the discrete increment of It\^{o} solution on
small interval $\left[ s,t\right] $.

Suppose $f:%
\mathbb{R}
^{e}\rightarrow L\left( 
\mathbb{R}
^{d},%
\mathbb{R}
^{e}\right) $ is $Lip\left( \beta \right) $ for $\beta >p$ and $\mathbf{\xi }%
\in T^{\left( n\right) }\left( 
\mathbb{R}
^{e}\right) $ for some integer $n\geq 2$. Denote $y^{i}:\left[ 0,T\right]
\rightarrow T^{\left( n\right) }\left( 
\mathbb{R}
^{e}\right) $, $i=1,2$, as the solution to the rough differential equations:%
\begin{eqnarray}
dy^{1} &=&f\left( \pi _{1}\left( y^{1}\right) \right) d\gamma \text{, \ }%
y_{s}^{1}=\mathbf{\xi }\in T^{\left( n\right) }\left( 
\mathbb{R}
^{e}\right) \text{,}  \label{inner RDE} \\
dy^{2} &=&f\left( \pi _{1}\left( y^{2}\right) \right) d\left( \gamma
+S_{2}\left( M\right) \right) \text{, \ }y_{s}^{2}=\mathbf{\xi }\in
T^{\left( n\right) }\left( 
\mathbb{R}
^{e}\right) \text{.}  \notag
\end{eqnarray}%
Denote $z^{i}:\left[ 0,T\right] \rightarrow T^{\left( 2\right) }\left( 
\mathbb{R}
^{e}\right) $, $i=1,2$, as the solution to the rough differential equations:%
\begin{eqnarray*}
dz^{1} &=&f\left( \pi _{1}\left( z^{1}\right) \right) d\gamma \text{, \ }%
z_{s}^{1}=\left( 1,\pi _{1}\left( \mathbf{\xi }\right) ,0\right) \in
T^{\left( 2\right) }\left( 
\mathbb{R}
^{e}\right) \text{,} \\
dz^{2} &=&f\left( \pi _{1}\left( z^{2}\right) \right) d\left( \gamma
+S_{2}\left( M\right) \right) \text{, \ }z_{s}^{2}=\left( 1,\pi _{1}\left( 
\mathbf{\xi }\right) ,0\right) \in T^{\left( 2\right) }\left( 
\mathbb{R}
^{e}\right) \text{.}
\end{eqnarray*}%
Then by using uniqueness of enhancement (i.e. Theorem \ref{Theorem of
enhancement}), it can be checked that (with $S_{n}\left( \cdot \right) $ in
Notation \ref{Notation Sn(X)}), $y_{u}^{i}=\mathbf{\xi }\otimes S_{n}\left(
z^{i}\right) _{s,u}$ and $y^{i}$ is the solution to the rough differential
equation driven by $z^{i}$: 
\begin{equation}
dy^{i}=1_{%
\mathbb{R}
^{e}}dz^{i}\text{, \ }y_{s}^{i}=\mathbf{\xi }\in T^{\left( n\right) }\left( 
\mathbb{R}
^{e}\right) \text{. }  \label{inner change of RDE}
\end{equation}

We want to modify the initial value of $y^{2}$ in $\left( \ref{inner RDE}%
\right) $ in such a way that, if we take expectation (of the modified $y^{2}$%
), we get $y^{1}$. However, the vector field $f$ in $\left( \ref{inner RDE}%
\right) $ may not be homogeneous w.r.t. scalar addition, so changing the
initial value may incur a corresponding change in the increment of the
solution path, which is not easy to cope with, especially when we want an
equality (i.e. expectation of modified $y^{2}$ equals to $y^{1}$). Instead,
we assume that $z^{1}$ and $z^{2}$ are the "Stratonovich" solutions which
are known, and consider the rough differential equations $\left( \ref{inner
change of RDE}\right) $ instead of $\left( \ref{inner RDE}\right) $.

Suppose that $z^{1}$ and $z^{2}$ (so $y^{1}$ and $y^{2}$) are known, which
we call the "Stratonovich" solutions. When the noise $M$ is present, we want
to modify the initial value of $y^{2}$ in $\left( \ref{inner change of RDE}%
\right) $\ in such a way that, we are expected to get the value $y_{t}^{1}$.
More specifically, we want to find the initial value $\delta _{s,t}\in
T^{\left( n\right) }\left( 
\mathbb{R}
^{e}\right) $, such that, the solution of the rough differential equation%
\begin{equation}
d\widetilde{y}^{2}=1_{%
\mathbb{R}
^{e}}dz^{2}\text{, \ }\widetilde{y}_{s}^{2}=\delta ^{s,t}\in T^{\left(
n\right) }\left( 
\mathbb{R}
^{e}\right) \text{,}  \label{inner RDE 1}
\end{equation}%
satisfies ($y_{s,t}:=y_{s}^{-1}\otimes y_{t}$)%
\begin{equation}
E\left( \widetilde{y}_{t}^{2}\right) =y_{t}^{1}\text{ i.e. }E\left( \delta
^{s,t}\otimes y_{s,t}^{2}\right) =\mathbf{\xi }\otimes y_{s,t}^{1}\text{ .}
\label{inner relationship}
\end{equation}

We let%
\begin{equation*}
\delta ^{s,t}:=\mathbf{\xi }\otimes y_{s,t}^{1}\otimes E\left(
y_{s,t}^{2}\right) ^{-1}\text{ (}\gamma \text{ is a fixed path).}
\end{equation*}%
Then we define the discrete It\^{o} increment (on $\left[ s,t\right] $) as
the "discounted" Stratonovich solution (at $t$): 
\begin{equation}
dy=1_{%
\mathbb{R}
^{e}}dz^{1}\text{, }y_{s}=\delta ^{s,t}\in T^{\left( n\right) }\left( 
\mathbb{R}
^{e}\right) \text{,}  \label{inner RDE 2}
\end{equation}%
which has the explicit expression: 
\begin{equation}
\mathbf{\xi }\otimes y_{s,t}^{1}\otimes E\left( y_{s,t}^{2}\right)
^{-1}\otimes y_{s,t}^{1}\text{.}  \label{discrete Ito increment}
\end{equation}%
When the noise is present, we consider $\left( \ref{inner RDE 1}\right) $
instead of $\left( \ref{inner RDE 2}\right) $, and (based on $\left( \ref%
{inner relationship}\right) $) we have $E\left( \widetilde{y}_{t}^{2}\right)
=y_{t}^{1}$. Thus, we define the discrete It\^{o} increment as a discounted
Stratonovich solution, in the sense that, when the noise is present, one is
expected to get the Stratonovich solution.

We concatenate the discrete It\^{o} increment in the form of $\left( \ref%
{discrete Ito increment}\right) $, let the mesh of partitions tends to zero,
and recover the solution of the rough differential equation%
\begin{equation*}
dy=f\left( \pi _{1}\left( y\right) \right) d\mathcal{I}_{2}\left( \gamma
,M\right) \text{,}
\end{equation*}%
where $\mathcal{I}_{2}\left( \gamma ,M\right) $ denotes the $p$-rough path
for some $p\in \lbrack 2,3)$: 
\begin{equation*}
\mathcal{I}_{2}\left( \gamma ,M\right) _{t}:=\left( 1,\pi _{1}\left( \gamma
_{t}\right) ,\pi _{2}\left( \gamma _{t}\right) -\frac{1}{2}%
\sum_{i,j=1}^{d}\left\langle M^{i},M^{j}\right\rangle _{t}e_{i}\otimes
e_{j}\right) \text{, \ }t\in \left[ 0,T\right] \text{.}
\end{equation*}

One might be tempted to replace the discrete increment $\left( \ref{discrete
Ito increment}\right) $ by the expectation of the solution of
forward-backward-forward equation, which, however, does not work, even on
the first level.

Our averaging process can be applied when $\gamma $ is a fixed $p$-rough
path, $p\in \lbrack 2,3)$, and $M=\int \phi dB$ with $\phi ~$a fixed path
taking value in $d\times d$ matrices and $B$ a $d$-dimensional Brownian
motion. When $\gamma =S_{2}\left( Z\right) $ is the Stratonovich signature
of a sample path of continuous local martingale $Z$, by setting $\phi
=\left\langle Z\right\rangle ^{\frac{1}{2}}$ we can recover the It\^{o}
solution in rough path.

\subsection{Rough path underlying}

\begin{definition}[perturbed rough path]
\label{Definition of perturbed rough path}Suppose $\gamma :\left[ 0,T\right]
\rightarrow \left( T^{\left( 2\right) }\left( 
\mathbb{R}
^{d}\right) ,\left\Vert \cdot \right\Vert \right) $ is a fixed $p$-rough
path on $\left[ 0,T\right] $ for some $p\in \lbrack 2,3)$, $\phi $ is a
fixed path defined on $\left[ 0,T\right] $ taking value in $d\times d$
matrices satisfying $\max_{1\leq i,j\leq d}\int_{0}^{T}\left( \phi
_{u}^{i,j}\right) ^{2}du<\infty $, and $B$ a $d$-dimensional Brownian
motion. Define continuous $d$-dimensional martingale $M$ as the It\^{o}
integral: 
\begin{equation}
M_{t}:=\int_{0}^{t}\phi _{u}dB_{u}\text{, }t\in \left[ 0,T\right] \text{.}
\label{Definition of M}
\end{equation}%
With $S_{2}\left( M\right) $ denotes the step-$2$ Stratonovich signature of $%
M$, we assume that $\gamma +S_{2}\left( M\right) :\left[ 0,T\right]
\rightarrow \left( T^{\left( 2\right) }\left( 
\mathbb{R}
^{d}\right) ,\left\Vert \cdot \right\Vert \right) $ is a $p$-rough process,
with the explicit expression: ($\gamma _{s,t}:=\gamma _{s}^{-1}\otimes
\gamma _{t}$)%
\begin{multline*}
\left( \gamma +S_{2}\left( M\right) \right) _{s,t}:=\left( 1,\pi _{1}\left(
\gamma _{s,t}\right) +M_{t}-M_{s},\pi _{2}\left( \gamma _{s,t}\right) \right.
\\
\left. +\iint_{s<u_{1}<u_{2}<t}\circ dM_{u_{1}}\otimes \circ
dM_{u_{2}}+R\left( \gamma ,M\right) ^{s,t}\right) \text{, }\forall 0\leq
s\leq t\leq T\text{,}
\end{multline*}%
where the cross term $R\left( \gamma ,M\right) $ satisfies%
\begin{equation}
E\left( R\left( \gamma ,M\right) ^{s,t}\right) =0\text{, }\forall 0\leq
s\leq t\leq T\text{.}  \label{Condition on expectation of signature}
\end{equation}
\end{definition}

Since $\gamma $ is fixed, the condition $\left( \ref{Condition on
expectation of signature}\right) $ is satisfied e.g. when the cross integral 
$R\left( \gamma ,M\right) $ is defined in classical It\^{o} sense or
Stratonovich sense, or when $R\left( \gamma ,M\right) $ is defined as the $%
L^{1}$ limit of piecewise-linear approximations.

\begin{definition}
Suppose $\gamma $ and $M$ are defined as in Definition \ref{Definition of
perturbed rough path}, $n\geq 1$ is an integer and $f:%
\mathbb{R}
^{e}\rightarrow L\left( 
\mathbb{R}
^{d},%
\mathbb{R}
^{e}\right) $ is $Lip\left( \beta \right) $ for $\beta >p$. For finite
partition $D=\left\{ t_{j}\right\} $ of $\left[ 0,T\right] $, define
piecewise-constant process $y^{n,D}\left( \gamma ,M\right) :\left[ 0,T\right]
\rightarrow T^{\left( n\right) }\left( 
\mathbb{R}
^{e}\right) $ as: ($y_{s,t}:=y_{s}^{-1}\otimes y_{t}$)%
\begin{align}
y^{n,D}\left( \gamma ,M\right) _{0}& :=\mathbf{\xi }\in T^{\left( n\right)
}\left( 
\mathbb{R}
^{e}\right) \text{,}  \label{Definition of y_n^D} \\
y^{n,D}\left( \gamma ,M\right) _{t}& :=y^{n,D}\left( \gamma ,M\right)
_{t_{j}}\otimes y_{t_{j},t_{j+1}}^{1,j}\otimes E\left(
y_{t_{j},t_{j+1}}^{2,j}\right) ^{-1}\otimes y_{t_{j},t_{j+1}}^{1,j}\text{,\ }%
t\in (t_{j},t_{j+1}]\text{,}  \notag
\end{align}%
where $y^{1,j}$ and $y^{2,j}$ denotes the solution of the rough differential
equations on $\left[ t_{j},t_{j+1}\right] $: 
\begin{eqnarray}
dy_{u}^{1,j} &=&f\left( \pi _{1}\left( y_{u}^{1,j}\right) \right) d\gamma
_{u}\text{, \ }y_{t_{j}}^{1,j}=y^{n,D}\left( \gamma ,M\right) _{t_{j}}\in
T^{\left( n\right) }\left( 
\mathbb{R}
^{e}\right) \text{,}  \label{Definition of y^(i,k)} \\
dy_{u}^{2,j} &=&f\left( \pi _{1}\left( y_{u}^{2,j}\right) \right) d\left(
\gamma +S_{2}\left( M\right) \right) _{u}\text{, \ }y_{t_{j}}^{2,j}=y^{n,D}%
\left( \gamma ,M\right) _{t_{j}}\in T^{\left( n\right) }\left( 
\mathbb{R}
^{e}\right) \text{.}  \notag
\end{eqnarray}
\end{definition}

It is worth noting that, (since $\gamma $ is fixed) $\left\{ y^{n,D}\left(
\gamma ,M\right) \right\} _{n,D}$ are deterministic.

\begin{theorem}
\label{Theorem general vector field for Ito RDE}Suppose $\gamma $ and $M$
are defined as in Definition \ref{Definition of perturbed rough path} and $%
p\in \lbrack 2,3)$. Denote $p$-rough path $\mathcal{I}_{2}\left( \gamma
,M\right) :\left[ 0,T\right] \rightarrow \left( T^{\left( 2\right) }\left( 
\mathbb{R}
^{d}\right) ,\left\Vert \cdot \right\Vert \right) $ as%
\begin{equation*}
\mathcal{I}_{2}\left( \gamma ,M\right) _{t}:=\left( 1,\pi _{1}\left( \gamma
_{t}\right) ,\pi _{2}\left( \gamma _{t}\right) -\frac{1}{2}%
\sum_{i,j=1}^{d}\left\langle M^{i},M^{j}\right\rangle _{t}e_{i}\otimes
e_{j}\right) \text{, }t\in \left[ 0,T\right] \text{.}
\end{equation*}%
Suppose $f:%
\mathbb{R}
^{e}\rightarrow L\left( 
\mathbb{R}
^{d},%
\mathbb{R}
^{e}\right) $ is $Lip\left( \beta \right) $ for $\beta >p$. If we assume
that, for some integer $n\geq 2$, 
\begin{equation}
E\left( \left\Vert \left( \gamma +S_{2}\left( M\right) \right) \right\Vert
_{p-var,\left[ 0,T\right] }^{np}\right) <\infty \text{,}
\label{condition of integrability}
\end{equation}%
then for $\mathbf{\xi }\in T^{\left( n\right) }\left( 
\mathbb{R}
^{e}\right) $, $y^{n,D}\left( \gamma ,M\right) $ (defined at $\left( \ref%
{Definition of y_n^D}\right) $) converge uniformly as $\left\vert
D\right\vert \rightarrow 0$ to the unique solution of the rough differential
equation%
\begin{equation}
dY=f\left( \pi _{1}\left( Y\right) \right) d\mathcal{I}_{2}\left( \gamma
,M\right) \text{, \ }Y_{0}=\mathbf{\xi }\in T^{\left( n\right) }\left( 
\mathbb{R}
^{e}\right) \text{.}  \label{RDE in consideration}
\end{equation}%
More specifically, 
\begin{equation}
\lim_{\left\vert D\right\vert \rightarrow 0}\max_{1\leq k\leq n}\sup_{0\leq
t\leq T}\left\vert \pi _{k}\left( y^{n,D}\left( \gamma ,M\right) _{t}\right)
-\pi _{k}\left( Y_{t}\right) \right\vert =0\text{.}
\label{uniform convergence in prob}
\end{equation}
\end{theorem}

The proof of Theorem \ref{Theorem general vector field for Ito RDE} starts
from page \pageref{Proof of Theorem general vector field for Ito RDE}.

\begin{remark}
\label{Remark convergence of first level}Based on the proof of Theorem \ref%
{Theorem general vector field for Ito RDE}, $E\left( \left\Vert \left(
\gamma +S_{2}\left( M\right) \right) \right\Vert _{p-var,\left[ 0,T\right]
}^{q}\right) <\infty $ for some $q>p$ is sufficient for the convergence of
the first level in $\left( \ref{uniform convergence in prob}\right) $.
\end{remark}

As a specific example where $\gamma $ is not a sample path of a martingale,
suppose $\left( X,B\right) $ is a $2d$-dimensional continuous Gaussian
process with independent components, and $B$ is a $d$-dimensional Brownian
motion. Further assume that the covariance function of $\left( X,B\right) $
has finite $\rho $-variation for some $\rho \in \lbrack 1,\frac{3}{2})$ (see
Section 15.3.2 \cite{Peter Friz}). Then based on Thm 15.33 \cite{Peter Friz}%
, our Theorem \ref{Theorem general vector field for Ito RDE} applies e.g. to 
$\left( \gamma ,B\right) $ with $\gamma =\left( \gamma ^{1},\gamma
^{2}\right) $ (Notation \ref{Notation X1 and X2}), where $\gamma ^{1}$ is
the step-$2$ Stratonovich signature of a sample path of fractional Brownian
motion with Hurst parameter $H>\frac{1}{3}$ and $\gamma ^{2}$ is a fixed
continuous path with finite $2^{-1}p$-variation for some $p\in \lbrack 2,3)$.

\subsection{Martingale underlying}

When $\gamma $ (in Definition \ref{Definition of perturbed rough path})\ is
a sample path of a continuous martingale, by choosing the right noise, we
can recover It\^{o} solution (i.e. solution of rough differential equation
driven by $\mathcal{I}_{2}\left( Z\right) $ defined at $\left( \ref%
{Definition In(Z)}\right) $ on page \pageref{Definition In(Z)}). The
following definition gives the explicit construction of the noise.

\begin{definition}
Suppose $Z$ is a continuous $d$-dimensional martingale in $L^{2}$\ on $\left[
0,T\right] $, and there exists a $d\times d$-matrices-valued adapted process 
$\psi \ $in $L^{2}$ on $\left[ 0,T\right] $\ such that%
\begin{equation*}
\left\langle Z\right\rangle _{t}=\int_{0}^{t}\psi _{s}^{T}\psi _{s}ds\text{, 
}\forall t\in \left[ 0,T\right] \text{, a.s..}
\end{equation*}%
Suppose $B$ is a $d$-dimensional Brownian motion, independent from $Z$. For
a fixed sample path of $Z$, define process $\widetilde{Z}$ as the It\^{o}
integral: 
\begin{equation}
\widetilde{Z}_{t}:=\int_{0}^{t}\psi _{s}dB_{s}\text{, }t\in \left[ 0,T\right]
\text{.}  \label{Definition of Ztilde}
\end{equation}
\end{definition}

The Corollary below follows from Theorem \ref{Theorem general vector field
for Ito RDE} and Remark \ref{Remark convergence of first level}, with proof
on page \pageref{Proof of Corollary of martingale underlying}.

\begin{corollary}
\label{Corollary for sample paths of martingale}Suppose $Z$ is a continuous $%
d$-dimensional martingale on $\left[ 0,T\right] $ in $L^{2n+\epsilon }$ for
some $\epsilon >0$ and integer $n\geq 1$. Suppose $f:%
\mathbb{R}
^{e}\rightarrow L\left( 
\mathbb{R}
^{d},%
\mathbb{R}
^{e}\right) $ is $Lip\left( \beta \right) $ for $\beta >2$. Denote $Y$ as
the solution to the rough differential equation: ($\mathcal{I}_{2}\left(
Z\right) $ defined at $\left( \ref{Definition In(Z)}\right) $ on page %
\pageref{Definition In(Z)}) 
\begin{equation*}
dY=f\left( \pi _{1}\left( Y\right) \right) d\mathcal{I}_{2}\left( Z\right) 
\text{, \ }Y_{0}=\mathbf{\xi }\in T^{\left( n\right) }\left( 
\mathbb{R}
^{e}\right) \text{.}
\end{equation*}%
Then for almost every sample path of $Z$, with $S_{2}\left( Z\right) $
denotes the step-$2$ Stratonovich signature of $Z$ and $\widetilde{Z}$
defined at $\left( \ref{Definition of Ztilde}\right) $, $y^{n,D}\left(
S_{2}\left( Z\right) ,\widetilde{Z}\right) $ (defined at $\left( \ref%
{Definition of y_n^D}\right) $) converge to $Y$ uniformly, i.e.%
\begin{equation*}
\lim_{\left\vert D\right\vert \rightarrow 0}\max_{1\leq k\leq n}\sup_{0\leq
t\leq T}\left\vert \pi _{k}\left( y^{n,D}\left( S_{2}\left( Z\right) ,%
\widetilde{Z}\right) _{t}\right) -\pi _{k}\left( Y_{t}\right) \right\vert =0%
\text{.}
\end{equation*}
\end{corollary}

\section{Proofs}

Our constants may implicitly depend on dimensions ($d$ and $e$). We specify
the dependence on other constants (e.g. $C_{p}$), but the exact value of
constants may change from line to line.

\subsection{Results from rough path}

We state the results in rough path theory that will be used in our proofs.

\noindent The Theorem below follows from Thm 14.12 in \cite{Peter Friz} and
Doob's maximal inequality.

\begin{theorem}
Suppose $M$ is a continuous $d$-dimensional local martingale. Denote $%
S_{2}\left( M\right) $ as the step-$2$ Stratonovich signature of $M$. Then
for any $q>0$ and any $p>2$, 
\begin{equation}
E\left( \left\Vert \left\langle M\right\rangle _{T}\right\Vert
^{2^{-1}q}\right) \sim E\left( \left\Vert \left\langle M\right\rangle
\right\Vert _{1-var,\left[ 0,T\right] }^{2^{-1}q}\right) \sim E\left(
\left\Vert S_{2}\left( M\right) \right\Vert _{p-var,\left[ 0,T\right]
}^{q}\right) \text{.}  \label{BDG and Doob for local}
\end{equation}%
If further assume that $M$ is an $L^{q}$ martingale for some $q>1$, then for
any $p>2$,%
\begin{equation}
E\left( \left\vert M_{T}-M_{0}\right\vert ^{q}\right) \sim E\left(
\left\Vert \left\langle M\right\rangle _{T}\right\Vert ^{2^{-1}q}\right)
\sim E\left( \left\Vert \left\langle M\right\rangle \right\Vert _{1-var, 
\left[ 0,T\right] }^{2^{-1}q}\right) \sim E\left( \left\Vert S_{2}\left(
M\right) \right\Vert _{p-var,\left[ 0,T\right] }^{q}\right) \text{.}
\label{BDG and Doob}
\end{equation}%
The equivalency in $\left( \ref{BDG and Doob for local}\right) $ and $\left( %
\ref{BDG and Doob}\right) $ are up to a positive constant depending on $p$, $%
q$ and $d$.
\end{theorem}

Suppose $\gamma $ is a $p$-rough path for $p\in \lbrack 2,3)$, denote $%
S_{n}\left( \gamma \right) $ as in Notation \ref{Notation Sn(X)} (on page %
\pageref{Notation Sn(X)}). Based on Thm 3.7 in \cite{Terrysnotes}, for any
integer $k\geq 1$, there exists constant $C_{p,k}$ such that, 
\begin{equation}
\left\vert \pi _{k}\left( S_{n}\left( \gamma \right) _{s,t}\right)
\right\vert \leq C_{p,k}\left\Vert \gamma \right\Vert _{p-var,\left[ s,t%
\right] }^{k}\text{, }\forall 0\leq s\leq t\leq T\text{.}
\label{estimates of higher level of extension with factorial decay}
\end{equation}%
As a consequence, $S_{n}\left( \gamma \right) :\left[ 0,T\right] \rightarrow
\left( T^{\left( n\right) }\left( 
\mathbb{R}
^{d}\right) ,\sum_{k=1}^{n}\left\vert \cdot \right\vert ^{\frac{1}{k}%
}\right) $ satisfies 
\begin{equation}
\left\Vert S_{n}\left( \gamma \right) \right\Vert _{p-var,\left[ s,t\right]
}\leq C_{p,n}\left\Vert \gamma \right\Vert _{p-var,\left[ s,t\right] }\text{%
, }\forall 0\leq s\leq t\leq T\text{.}
\label{estimates of p-variation of enhancement}
\end{equation}%
Moreover, for any integer $n\geq 2$ and any finite partition $D=\left\{
t_{k}\right\} _{k=0}^{K}$ of $\left[ s,t\right] \subseteq \left[ 0,T\right] $%
, we have (with $0\in \left( 
\mathbb{R}
^{d}\right) ^{\otimes \left( n+1\right) }$)

\begin{eqnarray}
&&\left\vert \pi _{n+1}\left( \left( S_{n}\left( \gamma \right)
_{t_{0},t_{1}}\oplus 0\right) \otimes \cdots \otimes \left( S_{n}\left(
\gamma \right) _{t_{K-1},t_{K}}\oplus 0\right) \right) -\pi _{n+1}\left(
S_{n+1}\left( \gamma \right) _{s,t}\right) \right\vert
\label{estimates of uniform convergence of multiplicative functional} \\
&\leq &C_{p,n}\sum_{k=0}^{K-1}\left\Vert \gamma \right\Vert _{p-var,\left[
t_{k},t_{k+1}\right] }^{n+1}\text{.}  \notag
\end{eqnarray}

Based on Prop 14.9 \cite{Peter Friz}, we have

\begin{theorem}
\label{Theorem regulartiy of Stratonovich signature of M}Suppose $%
M:[0,\infty )\rightarrow 
\mathbb{R}
^{d}$ is a continuous local martingale. Then the step-$2$ Stratonovich
signature of $M$ is a geometric $p$-rough process on $\left[ 0,T\right] $
for any $p>2$ and any $T>0$.
\end{theorem}

Suppose $\gamma =\left( \gamma ^{1},\gamma ^{2}\right) $ (Notation \ref%
{Notation X1 and X2} on p\pageref{Notation X1 and X2}) is a $p$-rough path
for some $p\in \lbrack 2,3)$ on $\left[ 0,T\right] $ taking value in $\left(
T^{\left( 2\right) }\left( 
\mathbb{R}
^{d}\right) ,\left\Vert \cdot \right\Vert \right) $. Then it is clear (or
based on \cite{Lejay and Victoir}) that,%
\begin{equation*}
\left\Vert \gamma \right\Vert _{p-var,\left[ s,t\right] }^{p}\leq \left\Vert
\gamma ^{1}\right\Vert _{p-var,\left[ s,t\right] }^{p}+\left\Vert \gamma
^{2}\right\Vert _{2^{-1}p-var,\left[ s,t\right] }^{2^{-1}p}\leq
C_{d}\left\Vert \gamma \right\Vert _{p-var,\left[ s,t\right] }^{p}\text{, }%
\forall 0\leq s\leq t\leq T\text{.}
\end{equation*}%
Based on Theorem 12.6 in \cite{Peter Friz} and Corollary 12.8 in \cite{Peter
Friz}, we have: (the second level estimation can be obtained by considering
the rough differential equation of the signature of the solution)

\begin{theorem}
\label{Theorem estimate of the solution of RDE}Suppose $\gamma \ $is a $p$%
-rough path for some $p\in \lbrack 2,3)$ on $\left[ 0,T\right] $ taking
value in $\left( T^{\left( 2\right) }\left( 
\mathbb{R}
^{d}\right) ,\left\Vert \cdot \right\Vert \right) $ and $f:%
\mathbb{R}
^{e}\rightarrow L\left( 
\mathbb{R}
^{d},%
\mathbb{R}
^{e}\right) $ is $Lip\left( \beta \right) $ for $\beta \in (p-1,2]$. If for
integer $n\geq 2$, $Y:\left[ 0,T\right] \rightarrow \left( T^{\left(
n\right) }\left( 
\mathbb{R}
^{e}\right) ,\left\Vert \cdot \right\Vert \right) $ is a solution to the
rough differential equation%
\begin{equation*}
dY=f\left( \pi _{1}\left( Y\right) \right) d\gamma \text{, \ }Y_{0}=\mathbf{%
\xi }\in T^{\left( n\right) }\left( 
\mathbb{R}
^{e}\right) \text{,}
\end{equation*}%
then for any $0\leq s\leq t\leq T$, we have ($Y_{s,t}:=Y_{s}^{-1}\otimes
Y_{t}$, $\gamma _{s,t}:=\gamma _{s}^{-1}\otimes \gamma _{t}$)%
\begin{gather}
\left\Vert Y\right\Vert _{p-var,\left[ s,t\right] }\leq C_{p,\beta
,f,n}\left( \left\Vert \gamma \right\Vert _{p-var,\left[ s,t\right] }\vee
\left\Vert \gamma \right\Vert _{p-var,\left[ s,t\right] }^{p}\right) \text{,}
\label{estimate of solution of RDE} \\
\left\vert \pi _{1}\left( Y_{s,t}\right) -f\left( \pi _{1}\left(
Y_{s}\right) \right) \pi _{1}\left( \gamma _{s,t}\right) -\left( Dff\right)
\left( \pi _{1}\left( Y_{s}\right) \right) \pi _{2}\left( \gamma
_{s,t}\right) \right\vert \leq C_{p,\beta ,f}\left\Vert \gamma \right\Vert
_{p-var,\left[ s,t\right] }^{\beta +1}\text{,}
\label{taylor estimate of the first level of solution of RDE} \\
\left\vert \pi _{2}\left( Y_{s,t}\right) -f\left( \pi _{1}\left(
Y_{s}\right) \right) \otimes f\left( \pi _{1}\left( Y_{s}\right) \right) \pi
_{2}\left( \gamma _{s,t}\right) \right\vert \leq C_{p,\beta ,f}\left\Vert
\gamma \right\Vert _{p-var,\left[ s,t\right] }^{\beta +1}\vee \left\Vert
\gamma \right\Vert _{p-var,\left[ s,t\right] }^{2p}\text{.}
\label{taylor estimate of the second level of the solution of RDE}
\end{gather}
\end{theorem}

The Theorem below follows from Theorem 12.10 in \cite{Peter Friz} and
Theorem 3.1.3 in \cite{TerryQian}.

\begin{theorem}
\label{Theorem continuity in initial value}Suppose $\gamma \ $is a $p$-rough
path for some $p\in \lbrack 2,3)$ on $\left[ 0,T\right] $ taking value in $%
\left( T^{\left( 2\right) }\left( 
\mathbb{R}
^{d}\right) ,\left\Vert \cdot \right\Vert \right) $, and $f:%
\mathbb{R}
^{e}\rightarrow L\left( 
\mathbb{R}
^{d},%
\mathbb{R}
^{e}\right) $ is $Lip\left( \beta \right) $ for $\beta >p$. Suppose $Y^{i}$, 
$i=1,2$, are the solution to the rough differential equations:%
\begin{equation*}
dY^{i}=f\left( \pi _{1}\left( Y^{i}\right) \right) d\gamma \text{, }%
Y_{0}^{i}=\mathbf{\xi }^{i}\in T^{\left( n\right) }\left( 
\mathbb{R}
^{e}\right) \text{.}
\end{equation*}%
Then ($Y_{s,t}^{i}:=\left( Y_{s}^{i}\right) ^{-1}\otimes Y_{t}^{i}$)%
\begin{eqnarray*}
&&\max_{1\leq k\leq n}\sup_{0\leq s\leq t\leq T}\left\vert \pi _{k}\left(
Y_{s,t}^{1}\right) -\pi _{k}\left( Y_{s,t}^{2}\right) \right\vert \\
&\leq &C_{p,\beta ,f}\left\vert \pi _{1}\left( \mathbf{\xi }^{1}\right) -\pi
_{1}\left( \mathbf{\xi }^{2}\right) \right\vert \exp \left( C_{p,\beta
,f}\left( \left\Vert \gamma \right\Vert _{p-var,\left[ 0,T\right]
}^{p}\right) \right) .
\end{eqnarray*}
\end{theorem}

\subsection{It\^{o} signature and relation with It\^{o} SDE}

\begin{lemma}
\label{Lemma group-valued martingale}Suppose $Z$ is a continuous local
martingale on $[0,\infty )$ taking value in $%
\mathbb{R}
^{d}$. For integer $n\geq 1$, denote $\mathcal{I}_{n}\left( Z\right) $ as at 
$\left( \ref{Definition In(Z)}\right) $ (on page \pageref{Definition In(Z)})
and $S_{n}\left( \cdot \right) $ as in\ Notation \ref{Notation Sn(X)} (on
page \pageref{Notation Sn(X)}). Then%
\begin{equation}
S_{n}\left( \mathcal{I}_{2}\left( Z\right) \right) _{t}=\mathcal{I}%
_{n}\left( Z\right) _{t}\text{, \ }t\in \lbrack 0,\infty )\text{, }\forall
n\geq 1\text{, a.s..}
\label{inner relation between enhancement and Ito signature}
\end{equation}
\end{lemma}

\begin{proof}
We prove $\left( \ref{inner relation between enhancement and Ito signature}%
\right) $ on $\left[ 0,T\right] $ for some $T>0$, for fixed $n\geq 1$ and
when $Z$ is a bounded continuous martingale. Then by properly stopping the
process and unionizing countably many null sets, we can prove $\left( \ref%
{inner relation between enhancement and Ito signature}\right) $.

Denote the filtration of $Z$ as $\left( \mathcal{F}_{t}\right) $. For $0\leq
s\leq t\leq T$, denote $S_{n}\left( \mathcal{I}_{2}\left( Z\right) \right)
_{s,t}:=S_{n}\left( \mathcal{I}_{2}\left( Z\right) \right) _{s}^{-1}\otimes
S_{n}\left( \mathcal{I}_{2}\left( Z\right) \right) _{t}$.

Firstly, we prove that $S_{n}\left( \mathcal{I}_{2}\left( Z\right) \right) $
is a group-valued martingale w.r.t. $\left( \mathcal{F}_{t}\right) $, i.e.
for any $0\leq s\leq t\leq T$, $E\left( S_{n}\left( \mathcal{I}_{2}\left(
Z\right) \right) _{s,t}|\mathcal{F}_{s}\right) =1\in T^{\left( n\right)
}\left( 
\mathbb{R}
^{d}\right) $. It is clear that, $S_{1}\left( \mathcal{I}_{2}\left( Z\right)
\right) =Z$ and $S_{2}\left( \mathcal{I}_{2}\left( Z\right) \right) =%
\mathcal{I}_{2}\left( Z\right) $ are martingales.

For integer $n\geq 2$, suppose $S_{n}\left( \mathcal{I}_{2}\left( Z\right)
\right) $ is a continuous martingale, we want to prove that $S_{n+1}\left( 
\mathcal{I}_{2}\left( Z\right) \right) $ is also a continuous martingale.
Based on $\left( \ref{estimates of uniform convergence of multiplicative
functional}\right) $ on page \pageref{estimates of uniform convergence of
multiplicative functional}, for any $0\leq s\leq t\leq T$, (with $0\in
\left( 
\mathbb{R}
^{d}\right) ^{\otimes \left( n+1\right) }$) 
\begin{eqnarray}
&&\pi _{n+1}\left( S_{n+1}\left( \mathcal{I}_{2}\left( Z\right) \right)
_{s,t}\right)  \label{inner convergence of discrete concatenations} \\
&=&\lim_{\left\vert D\right\vert \rightarrow 0,D=\left\{ t_{k}\right\}
\subset \left[ s,t\right] }\pi _{n+1}\left( \left( S_{n}\left( \mathcal{I}%
_{2}\left( Z\right) \right) _{t_{0},t_{1}}\oplus 0\right) \otimes \cdots
\otimes \left( S_{n}\left( \mathcal{I}_{2}\left( Z\right) \right)
_{t_{n-1},t_{n}}\oplus 0\right) \right) \text{,}  \notag
\end{eqnarray}%
and the error in $L^{1}$ is bounded by, for any $p>2$,%
\begin{eqnarray*}
&&C_{p,n}\lim_{\left\vert D\right\vert \rightarrow 0}E\left(
\sum_{k,t_{k}\in D}\left\Vert \mathcal{I}_{2}\left( Z\right) \right\Vert
_{p-var,\left[ t_{k},t_{k+1}\right] }^{n+1}\right) \\
&\leq &C_{p,n}\lim_{\left\vert D\right\vert \rightarrow 0}E\left( \left\Vert 
\mathcal{I}_{2}\left( Z\right) \right\Vert _{p-var,\left[ 0,T\right]
}^{n+1}\right) ^{\frac{p}{n+1}}E\left( \sup_{\left\vert t-s\right\vert \leq
\left\vert D\right\vert }\left\Vert \mathcal{I}_{2}\left( Z\right)
\right\Vert _{p-var,\left[ s,t\right] }^{n+1}\right) ^{1-\frac{p}{n+1}}\text{%
.}
\end{eqnarray*}%
Based on $\left( \ref{BDG and Doob}\right) $ and using that $Z$ is a bounded
martingale, we have%
\begin{eqnarray*}
E\left( \left\Vert \mathcal{I}_{2}\left( Z\right) \right\Vert _{p-var,\left[
0,T\right] }^{n+1}\right) &\leq &C_{n}\left( E\left( \left\Vert S_{2}\left(
Z\right) \right\Vert _{p-var,\left[ 0,T\right] }^{n+1}\right) +E\left(
\left\Vert \left\langle Z\right\rangle \right\Vert _{1-var,\left[ 0,T\right]
}^{\frac{n+1}{2}}\right) \right) \\
&\leq &C_{p,d,n}E\left( \left\vert Z_{T}-Z_{0}\right\vert ^{n+1}\right)
<\infty \text{.}
\end{eqnarray*}%
Thus, using dominated convergence theorem, the convergence at $\left( \ref%
{inner convergence of discrete concatenations}\right) $ is in $L^{1}$ and we
have:%
\begin{eqnarray*}
&&E\left( \pi _{n+1}\left( S_{n+1}\left( \mathcal{I}_{2}\left( Z\right)
\right) _{s,t}\right) \big|\mathcal{F}_{s}\right) \\
&=&\lim_{\left\vert D\right\vert \rightarrow 0,D=\left\{ t_{k}\right\}
_{k=0}^{K}\subset \left[ s,t\right] }\pi _{n+1}\left( E\left( \left(
S_{n}\left( \mathcal{I}_{2}\left( Z\right) \right) _{t_{0},t_{1}}\oplus
0\right) \otimes \cdots \otimes \left( S_{n}\left( \mathcal{I}_{2}\left(
Z\right) \right) _{t_{K-1},t_{K}}\oplus 0\right) \big|\mathcal{F}_{s}\right)
\right) \text{ a.s..}
\end{eqnarray*}%
While using the inductive hypothesis that $S_{n}\left( \mathcal{I}_{2}\left(
Z\right) \right) $ is a martingale, we have%
\begin{eqnarray*}
&&E\left( \left( S_{n}\left( \mathcal{I}_{2}\left( Z\right) \right)
_{t_{0},t_{1}}\oplus 0\right) \otimes \cdots \otimes \left( S_{n}\left( 
\mathcal{I}_{2}\left( Z\right) \right) _{t_{K-1},t_{K}}\oplus 0\right) \big|%
\mathcal{F}_{s}\right) \\
&=&E\left( \left( S_{n}\left( \mathcal{I}_{2}\left( Z\right) \right)
_{t_{0},t_{1}}\oplus 0\right) \otimes \cdots \otimes E\left( \left(
S_{n}\left( \mathcal{I}_{2}\left( Z\right) \right) _{t_{K-1},t_{K}}\oplus
0\right) \big|\mathcal{F}_{t_{K-1}}\right) \big|\mathcal{F}_{s}\right) \\
&=&E\left( \left( S_{n}\left( \mathcal{I}_{2}\left( Z\right) \right) \left(
Z\right) _{t_{0},t_{1}}\oplus 0\right) \otimes \cdots \otimes \left(
S_{n}\left( \mathcal{I}_{2}\left( Z\right) \right) _{t_{K-2},t_{K-1}}\oplus
0\right) \big|\mathcal{F}_{s}\right) \\
&=&\cdots =1\text{, a.s..}
\end{eqnarray*}%
Thus, for any $0\leq s\leq t\leq T$, (with $0\in \left( 
\mathbb{R}
^{d}\right) ^{\otimes \left( n+1\right) }$) 
\begin{eqnarray*}
E\left( S_{n+1}\left( \mathcal{I}_{2}\left( Z\right) \right) _{s,t}\big|%
\mathcal{F}_{s}\right) &=&E\left( \pi _{n+1}\left( S_{n+1}\left( \mathcal{I}%
_{2}\left( Z\right) \right) _{s,t}\right) \big|\mathcal{F}_{s}\right)
+E\left( S_{n}\left( \mathcal{I}_{2}\left( Z\right) \right) _{s,t}\oplus 0%
\text{ }\big|\mathcal{F}_{s}\right) \\
&=&1\in T^{\left( n+1\right) }\left( 
\mathbb{R}
^{d}\right) \text{.}
\end{eqnarray*}

Then we prove $\left( \ref{inner relation between enhancement and Ito
signature}\right) $. It is clear that $\left( \ref{inner relation between
enhancement and Ito signature}\right) $ holds for level $1$ and level $2$.
For the higher levels, we use mathematical induction. Suppose for some $%
k\geq 2$, we have%
\begin{eqnarray}
&&\left( 1,Z_{t}-Z_{0},\iint_{0<u_{1}<u_{2}<t}dZ_{u_{1}}\otimes
dZ_{u_{2}},\dots ,\idotsint\nolimits_{0<u_{1}<\cdots
<u_{k}<t}dZ_{u_{1}}\otimes \cdots \otimes dZ_{u_{k}}\right)
\label{inner inductive assumption} \\
&=&S_{k}\left( \mathcal{I}_{2}\left( Z\right) \right) _{t}\text{, }\forall
t\in \left[ 0,T\right] \text{, a.s..}  \notag
\end{eqnarray}%
Denote 
\begin{equation*}
Z_{s,t}^{It\hat{o},k}:=\idotsint\nolimits_{s<u_{1}<\cdots
<u_{k}<t}dZ_{u_{1}}\otimes \cdots \otimes dZ_{u_{k}}\text{, }\forall 0\leq
s\leq t\leq T\text{, }\forall k\geq 1\text{.}
\end{equation*}

Since $Z$ is a bounded continuous martingale, $t\mapsto Z_{0,t}^{It\hat{o}%
,k} $ is a continuous martingale adapted to the filtration of $Z$, and the
process $t\mapsto \left( Z_{0,t}^{It\hat{o},k},Z_{t}\right) $ is a
continuous martingale w.r.t. the filtration of $Z$. Then, based on Theorem %
\ref{Theorem regulartiy of Stratonovich signature of M} (on p\pageref%
{Theorem regulartiy of Stratonovich signature of M}), $\left( Z^{It\hat{o}%
,k},Z\right) $ can be enhanced by their Stratonovich integrals to a $p$%
-rough process for any $p>2$. As a result, we have%
\begin{equation*}
\sup_{D,D\subset \left[ 0,T\right] }\sum_{t_{j}\in D}\left\vert
\int_{t_{j}}^{t_{j+1}}\left( Z_{0,t}^{It\hat{o},k}-Z_{0,t_{j}}^{It\hat{o}%
,k}\right) \otimes \circ dZ_{t}\right\vert ^{\frac{p}{2}}<\infty \text{
a.s., }\forall p>2\text{.}
\end{equation*}%
Then using the relationship between It\^{o} integral and Stratonovich
integral, we have 
\begin{equation}
\sup_{D,D\subset \left[ 0,T\right] }\sum_{t_{j}\in D}\left\vert
\int_{t_{j}}^{t_{j+1}}\left( Z_{0,t}^{It\hat{o},k}-Z_{0,t_{j}}^{It\hat{o}%
,k}\right) \otimes dZ_{t}\right\vert ^{\frac{p}{2}}<\infty \text{ a.s., }%
\forall p>2\text{.}  \label{inner estimation for martingale 1}
\end{equation}%
Since the It\^{o} signature is a multiplicative functional, we have, for $%
t\in \left[ t_{j},t_{j+1}\right] $,%
\begin{equation}
Z_{0,t}^{It\hat{o},k}-Z_{0,t_{j}}^{It\hat{o},k}=%
\sum_{i=1}^{k-1}Z_{0,t_{j}}^{It\hat{o},k-i}\otimes Z_{t_{j},t}^{It\hat{o}%
,i}+Z_{t_{j},t}^{It\hat{o},k}\text{.}
\label{inner estimation for martingale 2}
\end{equation}%
For $i=1,\dots ,k-1$, we have%
\begin{equation*}
\sup_{D,D\subset \left[ 0,T\right] }\sum_{t_{j}\in D}\left\vert
\int_{t_{j}}^{t_{j+1}}Z_{0,t_{j}}^{It\hat{o},k-i}\otimes Z_{t_{j},t}^{It\hat{%
o},i}\otimes dZ_{t}\right\vert ^{\frac{p}{2}}\leq \sup_{t\in \left[ 0,T%
\right] }\left\vert Z_{0,t}^{It\hat{o},k-i}\right\vert ^{\frac{p}{2}%
}\sup_{D,D\subset \left[ 0,T\right] }\sum_{t_{j}\in D}\left\vert
\int_{t_{j}}^{t_{j+1}}Z_{t_{j},t}^{It\hat{o},i}\otimes dZ_{t}\right\vert ^{%
\frac{p}{2}}\text{.}
\end{equation*}%
Based on the inductive hypothesis $\left( \ref{inner inductive assumption}%
\right) $, we have, for any $p>2$ and $i=1,\dots ,k-1$, (since $i+1\geq 2$)%
\begin{equation}
\sup_{D,D\subset \left[ 0,T\right] }\sum_{t_{j}\in D}\left\vert
\int_{t_{j}}^{t_{j+1}}Z_{t_{j},t}^{It\hat{o},i}\otimes dZ_{t}\right\vert ^{%
\frac{p}{2}}\leq \left( \sup_{D,D\subset \left[ 0,T\right] }\sum_{t_{j}\in
D}\left\vert Z_{t_{j},t_{j+1}}^{It\hat{o},i+1}\right\vert ^{\frac{p}{i+1}%
}\right) ^{\frac{i+1}{2}}<\infty \text{ a.s., }\forall p>2\text{.}
\label{inner estimation for martingale 3}
\end{equation}%
Thus, combine $\left( \ref{inner estimation for martingale 1}\right) $, $%
\left( \ref{inner estimation for martingale 2}\right) $ and $\left( \ref%
{inner estimation for martingale 3}\right) $, we have%
\begin{equation}
\sup_{D,D\subset \left[ 0,T\right] }\sum_{t_{j}\in D}\left\vert
Z_{t_{j},t_{j+1}}^{It\hat{o},k+1}\right\vert ^{\frac{p}{2}}=\sup_{D,D\subset %
\left[ 0,T\right] }\sum_{t_{j}\in D}\left\vert
\int_{t_{j}}^{t_{j+1}}Z_{t_{j},t}^{It\hat{o},k}\otimes dZ_{t}\right\vert ^{%
\frac{p}{2}}<\infty \text{ a.s., }\forall p>2\text{.}
\label{inner estimation for martingale 4}
\end{equation}%
On the other hand, $\pi _{k+1}\left( S_{k+1}\left( \mathcal{I}_{2}\left(
Z\right) \right) _{t}\right) $ satisfies (based on $\left( \ref{estimates of
higher level of extension with factorial decay}\right) $ on page \pageref%
{estimates of higher level of extension with factorial decay}) 
\begin{equation}
\sup_{D,D\subset \left[ 0,T\right] }\sum_{t_{j}\in D}\left\vert \pi
_{k+1}\left( S_{k+1}\left( \mathcal{I}_{2}\left( Z\right) \right)
_{t_{j},t_{j+1}}\right) \right\vert ^{\frac{p}{k+1}}<\infty \text{ a.s., }%
\forall p>2\text{.}  \label{inner estimation for martingale 5}
\end{equation}%
Since both $\left( 1,Z,\dots ,Z^{It\hat{o},k},Z^{It\hat{o},k+1}\right) $ and 
$S_{k+1}\left( \mathcal{I}_{2}\left( Z\right) \right) $ are multiplicative
and $\left( \ref{inner inductive assumption}\right) $ holds, there exists a
process $\varphi :\left[ 0,T\right] \rightarrow \left( 
\mathbb{R}
^{d}\right) ^{\otimes \left( k+1\right) }$ such that%
\begin{equation}
Z_{s,t}^{It\hat{o},k+1}-\pi _{k+1}\left( S_{k+1}\left( \mathcal{I}_{2}\left(
Z\right) \right) _{s,t}\right) =\varphi _{t}-\varphi _{s}\text{, \ }\forall
0\leq s\leq t<\infty \text{, a.s..}  \label{inner expression of phi}
\end{equation}%
Moreover, $t\mapsto Z_{0,t}^{It\hat{o},k+1}$ is a continuous martingale,
and, as we proved above, $t\mapsto \pi _{k+1}\left( S_{k+1}\left( \mathcal{I}%
_{2}\left( Z\right) \right) _{0,t}\right) $ is also a continuous martingale.
Then based on $\left( \ref{inner expression of phi}\right) $, $t\mapsto
\left( \varphi _{t}-\varphi _{0}\right) $ is a continuous martingale
vanishing at $0$. Combined with $\left( \ref{inner estimation for martingale
4}\right) $ and $\left( \ref{inner estimation for martingale 5}\right) $
(with $k\geq 2$), we have 
\begin{equation*}
\sup_{D,D\subset \left[ 0,T\right] }\sum_{t_{j}\in D}\left\vert \varphi
_{t_{j+1}}-\varphi _{t_{j}}\right\vert ^{\frac{p}{2}}<\infty \text{ a.s.,
for any }p>2\text{.}
\end{equation*}%
Since a continuous martingale with finite $q$-variation for $q<2$ is a
constant, we have $\varphi _{t}\equiv \varphi _{0}$, and%
\begin{equation*}
Z_{t}^{It\hat{o},k+1}=\pi _{k+1}\left( S_{k+1}\left( \mathcal{I}_{2}\left(
Z\right) \right) _{t}\right) \text{, \ }\forall 0\leq t\leq T\text{, a.s..}
\end{equation*}
\end{proof}

\begin{proof}[Proof of Theorem \protect\ref{Theorem coincidence with
solution to SDE}]
\label{Proof of Theorem coincidence with solution to SDE}We only prove the
theorem when $Z$ is a continuous bounded martingale. Then by properly
stopping $Z$ and unionizing countably many null sets, we can prove Theorem %
\ref{Theorem coincidence with solution to SDE} for continuous local
martingales.

For $0\leq s\leq t\leq T$, we denote 
\begin{equation*}
Z_{s,t}^{1}:=Z_{t}-Z_{s}\text{, }Z_{s,t}^{2}:=\iint_{s<u_{1}<u_{2}<t}\circ
dZ_{u_{1}}\otimes \circ dZ_{u_{2}}\text{ and }\left\langle Z\right\rangle
_{s,t}:=\left\langle Z\right\rangle _{t}-\left\langle Z\right\rangle _{s}%
\text{.}
\end{equation*}%
Replace $\beta $ by $\beta \wedge 2$, and fix $p\in \left( 2,\beta +1\right) 
$. Define $\omega :\left\{ \left( s,t\right) |0\leq s\leq t\leq T\right\}
\rightarrow \overline{%
\mathbb{R}
^{+}}$ as 
\begin{equation*}
\omega \left( s,t\right) :=\left\Vert S_{2}\left( Z\right) \right\Vert
_{p-var,\left[ s,t\right] }^{p}+\left\Vert \left\langle Z\right\rangle
\right\Vert _{1-var,\left[ s,t\right] }^{\frac{p}{2}}\text{.}
\end{equation*}

Since $\beta >p-1$, the rough differential equation%
\begin{equation}
dY=f\left( \pi _{1}\left( Y\right) \right) d\mathcal{I}_{2}\left( Z\right) 
\text{, \ }Y_{0}=\mathbf{\xi }\in T^{\left( n\right) }\left( 
\mathbb{R}
^{e}\right) \text{,}  \label{Inner RDE driven by Ito signature}
\end{equation}%
has a solution. Denote $Y$ as a solution to the RDE $\left( \ref{Inner RDE
driven by Ito signature}\right) $, based on Euler estimate of solution of
RDE ($\left( \ref{taylor estimate of the first level of solution of RDE}%
\right) $ in Theorem \ref{Theorem estimate of the solution of RDE} on page %
\pageref{Theorem estimate of the solution of RDE}), we have the pathwise
estimate that, for any $0\leq s\leq t\leq T$, 
\begin{equation*}
\left\vert \pi _{1}\left( Y_{s,t}\right) -f\left( \pi _{1}\left(
Y_{s}\right) \right) Z_{s,t}^{1}-\left( Dff\right) \left( \pi _{1}\left(
Y_{s}\right) \right) \left( Z_{s,t}^{2}-\frac{1}{2}\left\langle
Z\right\rangle _{s,t}\right) \right\vert \leq C_{p,\beta ,f}\omega \left(
s,t\right) ^{\frac{\beta +1}{p}}\text{.}
\end{equation*}%
As a result, for almost every sample path of $Z$, any solution $Y$ to $%
\left( \ref{Inner RDE driven by Ito signature}\right) $ and any $0\leq s\leq
t\leq T$,%
\begin{eqnarray}
&&\pi _{1}\left( Y_{s,t}\right)  \label{inner limit} \\
&=&\lim_{\left\vert D\right\vert \rightarrow 0,D\subset \left[ s,t\right]
}\sum_{t_{k}\in D}\left( f\left( \pi _{1}\left( Y_{t_{k}}\right) \right)
Z_{t_{k},t_{k+1}}^{1}+\left( Dff\right) \left( \pi _{1}\left(
Y_{t_{k}}\right) \right) \left( Z_{t_{k},t_{k+1}}^{2}-\frac{1}{2}%
\left\langle Z\right\rangle _{t_{k},t_{k+1}}\right) \right) \text{.}  \notag
\end{eqnarray}%
Moreover, we have 
\begin{equation}
\lim_{\left\vert D\right\vert \rightarrow 0,D\subset \left[ s,t\right]
}\sum_{t_{k}\in D}\left( Dff\right) \left( \pi _{1}\left( Y_{t_{k}}\right)
\right) \left( Z_{t_{k},t_{k+1}}^{2}-\frac{1}{2}\left\langle Z\right\rangle
_{t_{k},t_{k+1}}\right) =0\text{ in }L^{2}\text{.}
\label{inner convergence in L2}
\end{equation}%
Indeed, since RDE solution is the limit of ODE solutions, and solution of
ODE can be recovered via Picard iteration ($f$ is $Lip\left( \beta \right) $
for $\beta >1$), $\pi _{1}\left( Y\right) $ is adapted to the filtration of $%
Z$. Since $\left( Dff\right) \left( \pi _{1}\left( Y\right) \right) $ is in $%
L^{4}$ (actually bounded, here we use $L^{4}$ for the convenience of second
level) and $Z$ is bounded, the cross terms in $L^{2}$ norm of $\left( \ref%
{inner convergence in L2}\right) $ vanish after taking conditional
expectation. Thus, for any $1\leq i,j\leq d$ and any $1\leq s\leq e$, we
have ($\left( Dff\right) ^{s}$ denotes the projection to the $s$th
coordinate of $%
\mathbb{R}
^{e}$) 
\begin{eqnarray*}
&&E\left( \left( \sum_{t_{k}\in D}\left( Dff\right) ^{s}\left( \pi
_{1}\left( Y_{t_{k}}\right) \right) \left( Z_{t_{k},t_{k+1}}^{2,i,j}-\frac{1%
}{2}\left\langle Z^{i},Z^{j}\right\rangle _{t_{k},t_{k+1}}\right) \right)
^{2}\right) \\
&=&\sum_{t_{k}\in D}E\left( \left( \left( Dff\right) ^{s}\left( \pi
_{1}\left( Y_{t_{k}}\right) \right) \left( Z_{t_{k},t_{k+1}}^{2,i,j}-\frac{1%
}{2}\left\langle Z^{i},Z^{j}\right\rangle _{t_{k},t_{k+1}}\right) \right)
^{2}\right) \\
&\leq &\sup_{t\in \left[ 0,T\right] }E\left( \left\vert \left( Dff\right)
^{s}\left( \pi _{1}\left( Y_{t}\right) \right) \right\vert ^{4}\right) ^{%
\frac{1}{2}}\sum_{t_{k}\in D}E\left( \left( Z_{t_{k},t_{k+1}}^{2,i,j}-\frac{1%
}{2}\left\langle Z^{i},Z^{j}\right\rangle _{t_{k},t_{k+1}}\right)
^{4}\right) ^{\frac{1}{2}} \\
&\leq &\sup_{t\in \left[ 0,T\right] }E\left( \left\vert \left( Dff\right)
^{s}\left( \pi _{1}\left( Y_{t}\right) \right) \right\vert ^{4}\right) ^{%
\frac{1}{2}}\sum_{t_{k}\in D}E\left( \left\Vert S_{2}\left( Z\right)
\right\Vert _{3-var,\left[ t_{k},t_{k+1}\right] }^{8}+\left\Vert
\left\langle Z\right\rangle \right\Vert _{1-var,\left[ t_{k},t_{k+1}\right]
}^{4}\right) ^{\frac{1}{2}}\text{ \ then use }\left( \ref{BDG and Doob}%
\right) \\
&\leq &C_{d}\sup_{t\in \left[ 0,T\right] }E\left( \left\vert \left(
Dff\right) ^{s}\left( \pi _{1}\left( Y_{t}\right) \right) \right\vert
^{4}\right) ^{\frac{1}{2}}\sum_{t_{k}\in D}E\left( \left\Vert \left\langle
Z\right\rangle \right\Vert _{1-var,\left[ t_{k},t_{k+1}\right] }^{4}\right)
^{\frac{1}{2}} \\
&\leq &C_{d}\left( \sup_{t\in \left[ 0,T\right] }E\left( \left\vert \left(
Dff\right) ^{s}\left( \pi _{1}\left( Y_{t}\right) \right) \right\vert
^{4}\right) ^{\frac{1}{2}}\right) E\left( \left\Vert \left\langle
Z\right\rangle \right\Vert _{1-var,\left[ 0,T\right] }^{4}\right) ^{\frac{1}{%
8}}E\left( \sup_{\left\vert t-s\right\vert \leq \left\vert D\right\vert
}\left\Vert \left\langle Z\right\rangle \right\Vert _{1-var,\left[ s,t\right]
}^{4}\right) ^{\frac{3}{8}}\text{.}
\end{eqnarray*}%
Since $Z$ is bounded, by using $\left( \ref{BDG and Doob}\right) $ on p%
\pageref{BDG and Doob}, we have%
\begin{equation*}
E\left( \left\Vert \left\langle Z\right\rangle \right\Vert _{1-var,\left[ 0,T%
\right] }^{4}\right) \leq E\left( \left\vert Z_{T}-Z_{0}\right\vert
^{8}\right) <\infty \text{.}
\end{equation*}%
Then by dominated convergence theorem, $\left( \ref{inner convergence in L2}%
\right) $ holds.

On the other hand, since $f\left( \pi _{1}\left( Y_{t}\right) \right) ~$is
bounded and adapted, and $Z$ is bounded, by using $\left( \ref{BDG and Doob}%
\right) $ on p\pageref{BDG and Doob}, we have%
\begin{equation*}
E\left( \int_{0}^{T}\left\vert f\left( \pi _{1}\left( Y_{u}\right) \right)
\right\vert ^{2}d\left\langle Z\right\rangle _{u}\right) \leq \left\vert
f\right\vert _{Lip\left( \beta \right) }^{2}E\left( \left\Vert \left\langle
Z\right\rangle \right\Vert _{1-var,\left[ 0,T\right] }\right) \leq
\left\vert f\right\vert _{Lip\left( \beta \right) }^{2}E\left( \left\vert
Z_{T}-Z_{0}\right\vert ^{2}\right) ^{\frac{1}{2}}<\infty \text{.}
\end{equation*}%
Thus,%
\begin{equation*}
\lim_{\left\vert D\right\vert \rightarrow 0,D\subset \left[ s,t\right]
}\sum_{t_{k}\in D}f\left( \pi _{1}\left( Y_{t_{k}}\right) \right)
Z_{t_{k},t_{k+1}}^{1}
\end{equation*}%
converge in $L^{2}$ to the It\^{o} integral $\int_{s}^{t}f\left( \pi
_{1}\left( Y_{u}\right) \right) dZ_{u}$. As a result, $\left( \ref{inner
limit}\right) $ converge a.s. and in $L^{2}$, with $\pi _{1}\left(
Y_{s,t}\right) $ the a.s. limit and the It\^{o} integral $%
\int_{s}^{t}f\left( \pi _{1}\left( Y_{u}\right) \right) dZ_{u}$ the $L^{2}$
limit. Therefore, since the null set can be chosen to be independent from $s$
and $t$, we have 
\begin{equation*}
\pi _{1}\left( Y_{s,t}\right) =\int_{s}^{t}f\left( \pi _{1}\left(
Y_{u}\right) \right) dZ_{u}\text{, \ }\forall 0\leq s\leq t\leq T\text{,
a.s..}
\end{equation*}%
Thus, $\pi _{1}\left( Y\right) $ is a strong continuous solution to the
stochastic differential equation:%
\begin{equation}
dy=f\left( y\right) dZ\text{, \ }y_{0}=\pi _{1}\left( \mathbf{\xi }\right)
\in 
\mathbb{R}
^{e}\text{.}  \label{SDE driven by continuous martingale}
\end{equation}%
On the other hand, since $f$ is $Lip\left( \beta \right) $, $\beta >1$,
based on Thm (2.1) on p375 \cite{Revuz and Yor}, the SDE $\left( \ref{SDE
driven by continuous martingale}\right) $ has a unique strong continuous
solution (denoted as $y$). Thus, $\pi _{1}\left( Y\right) $ exists uniquely
a.s. and%
\begin{equation*}
\pi _{1}\left( Y_{t}\right) =y_{t}\text{, }t\in \left[ 0,T\right] \text{,
a.s..}
\end{equation*}

For $\pi _{2}\left( Y\right) $, consider the rough differential equation on $%
\left[ 0,T\right] $:%
\begin{eqnarray}
dU_{t} &=&\left( f\left( \pi _{%
\mathbb{R}
^{e}}\left( U_{t}\right) \right) ,\left( \pi _{%
\mathbb{R}
^{e}}\left( U_{t}\right) -\pi _{%
\mathbb{R}
^{e}}\left( U_{0}\right) \right) \otimes f\left( \pi _{%
\mathbb{R}
^{e}}\left( U_{t}\right) \right) \right) d\mathcal{I}_{2}\left( Z\right) _{t}%
\text{, }  \label{inner RDE equations} \\
U_{0} &=&\left( 1,\left( \pi _{1}\left( \mathbf{\xi }\right) ,0\right)
,0\right) \in T^{\left( 2\right) }\left( 
\mathbb{R}
^{e}\oplus \left( 
\mathbb{R}
^{e}\right) ^{\otimes 2}\right) \text{.}  \notag
\end{eqnarray}%
Although the vector field for $U$ in $\left( \ref{inner RDE equations}%
\right) $ is only locally $Lip\left( \beta \right) $, $\beta >1$, its first
level has a global solution 
\begin{equation*}
\pi _{1}\left( U_{t}\right) =\left( \pi _{1}\left( Y_{t}\right) ,\pi
_{2}\left( Y_{0,t}\right) \right) \text{, \ }t\in \left[ 0,T\right] \text{.}
\end{equation*}%
By using $\left( \ref{estimate of solution of RDE}\right) $ in Theorem \ref%
{Theorem estimate of the solution of RDE} (on p\pageref{Theorem estimate of
the solution of RDE}), $\left( \ref{BDG and Doob}\right) $ on p\pageref{BDG
and Doob} and that $Z$ is bounded, one can prove that, for any $p>2$, (we
concentrate on the projection of $\pi _{1}\!\left( U\right) $ to $\left( 
\mathbb{R}
^{e}\right) ^{\otimes 2}$) 
\begin{eqnarray*}
&&\sup_{t\in \left[ 0,T\right] }E\left( \left\vert f\left( \pi _{1}\left(
Y_{t}\right) \right) \otimes f\left( \pi _{1}\left( Y_{t}\right) \right)
+\left( \pi _{1}\left( Y_{t}\right) -\pi _{1}\left( Y_{0}\right) \right)
\otimes \left( Dff\right) \left( \pi _{1}\left( Y_{t}\right) \right)
\right\vert ^{4}\right) \\
&\leq &8\left\vert f\right\vert _{Lip\left( \beta \right) }^{8}\left(
1+E\left( \left\Vert Y\right\Vert _{p-var,\left[ 0,T\right] }^{4}\right)
\right) <\infty \text{,}
\end{eqnarray*}%
and 
\begin{equation*}
E\left( \int_{0}^{T}\left\vert \pi _{1}\left( Y_{u}\right) -\pi _{1}\left(
Y_{0}\right) \right\vert ^{2}\left\vert f\right\vert _{Lip\left( \beta
\right) }^{2}d\left\langle Z\right\rangle _{u}\right) \leq \left\vert
f\right\vert _{Lip\left( \beta \right) }^{2}E\left( \left\Vert Y\right\Vert
_{p-var,\left[ 0,T\right] }^{2}\left\Vert \left\langle Z\right\rangle
\right\Vert _{1-var,\left[ 0,T\right] }\right) <\infty \text{.}
\end{equation*}%
Thus, by following similar reasoning as above, we have that $\pi _{2}\left(
Y_{0,t}\right) $ exists uniquely a.s. as a continuous martingale w.r.t. the
filtration of $Z$. On the other hand, consider $\alpha :\left\{ \left(
s,t\right) |0\leq s\leq t\leq T\right\} \rightarrow \left( 
\mathbb{R}
^{e}\right) ^{\otimes 2}$ defined as: ($y$ denotes the solution of SDE $%
\left( \ref{SDE driven by continuous martingale}\right) $)%
\begin{equation*}
\alpha _{s,t}:=\iint\nolimits_{s<u_{1}<u_{2}<t}dy_{u_{1}}\otimes dy_{u_{2}}%
\text{, }\forall 0\leq s\leq t\leq T\text{.}
\end{equation*}%
Then the process $t\mapsto \alpha _{0,t}$ is a continuous martingale w.r.t
the filtration of $Z$. Since both $\left( \pi _{1}\left( Y\right) ,\pi
_{2}\left( Y\right) \right) $ and $\left( y,\alpha \right) $ are
multiplicative and $\pi _{1}\left( Y\right) $ equals to $y$ a.s., there
exists a process $\varphi $ on $\left[ 0,T\right] $ such that%
\begin{equation*}
\varphi _{t}-\varphi _{s}=\pi _{2}\left( Y_{s,t}\right) -\alpha _{s,t}\text{%
, }\forall 0\leq s\leq t\leq T\text{, a.s..}
\end{equation*}%
On the other hand, since both $t\mapsto \pi _{2}\left( Y_{0,t}\right) $ and $%
t\mapsto \alpha _{0,t}$ are continuous martingales w.r.t. the filtration of $%
Z$, $t\mapsto \varphi _{t}-\varphi _{0}$ is also a continuous martingale
vanishing at $0$. Moreover, we have, (based on $\left( \ref{estimate of
solution of RDE}\right) $ in\ Theorem \ref{Theorem estimate of the solution
of RDE} on p\pageref{Theorem estimate of the solution of RDE}) 
\begin{equation}
\sup_{D\subset \left[ 0,T\right] }\sum_{k,t_{k}\in D}\left\vert \pi
_{2}\left( Y_{t_{k},t_{k+1}}\right) \right\vert ^{\frac{p}{2}}<\infty \text{
a.s., }\forall p>2\text{.}  \label{inner estimation of p-var of pi_2(Y)}
\end{equation}%
and $\alpha $ satisfies (Theorem \ref{Theorem regulartiy of Stratonovich
signature of M} on p\pageref{Theorem regulartiy of Stratonovich signature of
M})%
\begin{equation}
\sup_{D\subset \left[ 0,T\right] }\sum_{k,t_{k}\in D}\left\vert \alpha
_{t_{k},t_{k+1}}\right\vert ^{\frac{p}{2}}<\infty \text{ a.s., }\forall p>2.
\label{inner estimation of p-var of alpha}
\end{equation}%
Therefore, $t\mapsto \left( \varphi _{t}-\varphi _{0}\right) $ is a
continuous martingale vanishing at $0$ with finite $2^{-1}p$-variation for
any $p>2$. Thus, we have 
\begin{eqnarray*}
\varphi _{t} &\equiv &\varphi _{0}\text{, \ }\forall 0\leq t\leq T\text{,
a.s.} \\
\text{and }\pi _{2}\left( Y_{s,t}\right) &=&\alpha _{s,t}\text{, \ }\forall
0\leq s\leq t\leq T\text{, a.s..}
\end{eqnarray*}%
Therefore, the solution $Y$ of the RDE $\left( \ref{Inner RDE driven by Ito
signature}\right) $ has the explicit expression: 
\begin{equation*}
S_{2}\left( Y\right) _{0,t}=\left( 1,y_{t}-\pi _{1}\left( \mathbf{\xi }%
\right) ,\iint\nolimits_{0<u_{1}<u_{2}<t}dy_{u_{1}}\otimes dy_{u_{2}}\right) 
\text{, }t\in \left[ 0,T\right] \text{, a.s..}
\end{equation*}

Using the Theorem of enhancement (Theorem \ref{Theorem of enhancement} on p%
\pageref{Theorem of enhancement}) and Lemma \ref{Lemma group-valued
martingale} on p\pageref{Lemma group-valued martingale}, we have the
explicit expression that%
\begin{equation*}
S_{n}\left( Y\right) _{t}=\mathbf{\xi }\otimes S_{n}\left( Y\right) _{0,t}=%
\mathbf{\xi }\otimes S_{n}\left( \mathcal{I}_{2}\left( y\right) \right)
_{0,t}=\mathbf{\xi }\otimes \mathcal{I}_{n}\left( y\right) _{0,t}\text{, }%
t\in \left[ 0,T\right] \text{, }\forall n\geq 1\text{, a.s.,}
\end{equation*}%
where $y$ denotes the solution to the It\^{o} stochastic differential
equation $\left( \ref{SDE driven by continuous martingale}\right) $ and $%
\mathcal{I}_{n}\left( y\right) $ is the It\^{o} signature of $y$ defined at $%
\left( \ref{Definition In(Z)}\right) $ on p\pageref{Definition In(Z)}.
\end{proof}

\begin{lemma}
\label{Lemma strong convergence of bracket process}Suppose $Z$ is a
continuous local martingale on $[0,\infty )$ taking value in $%
\mathbb{R}
^{d}$. For $T>0$ and finite partition $D=\left\{ t_{k}\right\} _{k=0}^{n}$
of $\left[ 0,T\right] $, define $\left\langle Z\right\rangle ^{D}:\left[ 0,T%
\right] \rightarrow \left( 
\mathbb{R}
^{d}\right) ^{\otimes 2}$ as%
\begin{equation}
\left\langle Z\right\rangle _{0}^{D}=0\text{ and }\left\langle
Z\right\rangle _{t}^{D}:=\frac{t-t_{k}}{t_{k+1}-t_{k}}\left(
Z_{t_{k+1}}-Z_{t_{k}}\right) ^{\otimes 2}+\left\langle Z\right\rangle
_{t_{k}}^{D}\text{, }t\in \left[ t_{k},t_{k+1}\right] \text{, }t_{k}\in D%
\text{,}  \label{definition of bracket Z^D}
\end{equation}%
Then for any $p>2$, 
\begin{equation}
\lim_{\left\vert D\right\vert \rightarrow 0,D\subset \left[ 0,T\right]
}\left\Vert \left\langle Z\right\rangle ^{D}-\left\langle Z\right\rangle
\right\Vert _{\frac{p}{2}-var,\left[ 0,T\right] }=0\text{ in prob.,}
\label{strong convergence of bracket process}
\end{equation}%
and the convergence in $\left( \ref{strong convergence of bracket process}%
\right) $\ is in $L^{\frac{q}{2}}$ if $\left\langle Z\right\rangle _{T}$ is
in $L^{\frac{q}{2}}$ for some $q\geq 1$.
\end{lemma}

\begin{proof}
We assume that $\left\langle Z\right\rangle _{T}$ is in $L^{\frac{q}{2}}$
for some $q\geq 1$. Then by properly stopping $Z$, one can prove $\left( \ref%
{strong convergence of bracket process}\right) $ for continuous local
martingales.

For finite partition $D=\left\{ t_{k}\right\} _{k=0}^{n}\subset \left[ 0,T%
\right] $, we define discrete process%
\begin{equation}
\lambda \!\left( D\right) _{0}=0\text{ and }\lambda \!\left( D\right)
_{t_{k}}=\sum_{j=0}^{k-1}\left( Z_{t_{j+1}}-Z_{t_{j}}\right) ^{\otimes
2}-\left\langle Z\right\rangle _{t_{k}}\text{, }t_{k}\in D\text{, }k\geq 1%
\text{.}  \label{inner definition of lamdaD}
\end{equation}%
Then, with $\left( \mathcal{F}_{s}\right) $ denoting the filtration of $Z$, $%
\lambda \!\left( D\right) $ is a discrete martingale w.r.t. $\left( \mathcal{%
F}_{t_{k}}\right) _{k=0}^{n}$.

It can be checked that $\lambda \!\left( D\right) _{t_{k}}=\left\langle
Z\right\rangle _{t_{k}}^{D}-\left\langle Z\right\rangle _{t_{k}}$, $\forall
t_{k}\in D$. For any $0\leq s\leq t\leq T$, we decompose $\left[ s,t\right] $
as $[s,t_{k_{1}})\sqcup \lbrack t_{k_{1}},t_{k_{2}})\sqcup \left[ t_{k_{2}},t%
\right] $ with $t_{k_{1}-1}<s$ and $t_{k_{2}+1}>t$, and get%
\begin{align}
& \left\Vert \left\langle Z\right\rangle ^{D}-\left\langle Z\right\rangle
\right\Vert _{\frac{p}{2}-var,\left[ 0,T\right] }^{\frac{q}{2}}  \notag \\
& \leq C_{p,q}\left( \left\Vert \lambda \!\left( D\right) \right\Vert _{%
\frac{p}{2}-var,\left[ 0,T\right] }^{\frac{q}{2}}+\left(
\sum_{k=0}^{n-1}\left\Vert \left\langle Z\right\rangle ^{D}\right\Vert _{%
\frac{p}{2}-var,\left[ t_{k},t_{k+1}\right] }^{\frac{p}{2}}\right) ^{\frac{q%
}{p}}+\left( \sum_{k=0}^{n-1}\left\Vert \left\langle Z\right\rangle
\right\Vert _{\frac{p}{2}-var,\left[ t_{k},t_{k+1}\right] }^{\frac{p}{2}%
}\right) ^{\frac{q}{p}}\right)  \label{inner separate estimate}
\end{align}%
Then we estimate the three terms in $\left( \ref{inner separate estimate}%
\right) $ separately.

For the first term in $\left( \ref{inner separate estimate}\right) $, based
on the definition of $\lambda \!\left( D\right) $ at $\left( \ref{inner
definition of lamdaD}\right) $, we have ($Z=\left( Z^{1},Z^{2},\dots
,Z^{d}\right) $) 
\begin{eqnarray}
&&E\left( \left\Vert \lambda \!\left( D\right) \right\Vert _{1-var,\left[ 0,T%
\right] }^{\frac{q}{2}}\right)  \label{inner bound on 1-var 1} \\
&\leq &C_{d,q}\left( \sum_{i=1}^{d}E\left( \left( \sum_{k=0}^{n-1}\left(
Z_{t_{k+1}}^{i}-Z_{t_{k}}^{i}\right) ^{2}\right) ^{\frac{q}{2}}\right)
+E\left( \left\Vert \left\langle Z\right\rangle \right\Vert _{1-var,\left[
0,T\right] }^{\frac{q}{2}}\right) \right)  \notag \\
&\leq &C_{d,q}\left( \sum_{i=1}^{d}E\left( \left\vert \sum_{k=0}^{n-1}\left(
Z_{t_{k+1}}^{i}-Z_{t_{k}}^{i}\right) ^{2}-\left\langle Z^{i}\right\rangle
_{T}\right\vert ^{\frac{q}{2}}\right) +E\left( \left\Vert \left\langle
Z\right\rangle \right\Vert _{1-var,\left[ 0,T\right] }^{\frac{q}{2}}\right)
\right) \text{.}  \notag
\end{eqnarray}%
For $i=1,2,\dots ,d$,\ the discrete processes $\lambda ^{i}\!\left( D\right) 
$ defined as 
\begin{equation*}
\lambda ^{i}\!\left( D\right) _{0}=0\text{ and }\lambda ^{i}\!\left(
D\right) _{t_{k}}:=\sum_{j=0}^{k}\left( \left(
Z_{t_{j+1}}^{i}-Z_{t_{j}}^{i}\right) ^{2}-\left\langle Z^{i}\right\rangle
_{t_{j},t_{j+1}}\right) \text{, }t_{k}\in D\text{, }k\geq 1\text{,}
\end{equation*}%
are discrete martingales w.r.t. $\left( \mathcal{F}_{t_{k}}\right)
_{k=0}^{n} $. Thus, by using BDG inequality, we have%
\begin{eqnarray}
E\left( \left\vert \sum_{k=0}^{n-1}\left(
Z_{t_{k+1}}^{i}-Z_{t_{k}}^{i}\right) ^{2}-\left\langle Z^{i}\right\rangle
_{0,T}\right\vert ^{\frac{q}{2}}\right) &\leq &E\left( \left\Vert \lambda
^{i}\!\left( D\right) \right\Vert _{\infty -var,\left[ 0,T\right] }^{\frac{q%
}{2}}\right) \leq C_{q}E\left( \left\vert \left\langle \lambda ^{i}\!\left(
D\right) \right\rangle _{T}\right\vert ^{\frac{q}{2}}\right)
\label{inner discrete BDG} \\
&\leq &C_{q}E\left( \left( \sum_{k=0}^{n-1}\left( \left(
Z_{t_{k+1}}^{i}-Z_{t_{k}}^{i}\right) ^{4}+\left\langle Z^{i}\right\rangle
_{t_{k},t_{k+1}}^{2}\right) \right) ^{\frac{q}{4}}\right)  \notag \\
&\leq &C_{q}\left( E\left( \left\Vert Z\right\Vert _{4-var,\left[ 0,T\right]
}^{q}\right) +E\left( \left\Vert \left\langle Z\right\rangle \right\Vert
_{1-var,\left[ 0,T\right] }^{\frac{q}{2}}\right) \right) \text{.}  \notag
\end{eqnarray}%
Thus, by using $\left( \ref{inner discrete BDG}\right) $ and $\left( \ref%
{BDG and Doob for local}\right) $ (on page \pageref{BDG and Doob}), we
continue with $\left( \ref{inner bound on 1-var 1}\right) $, and get%
\begin{eqnarray}
E\left( \left\Vert \lambda \!\left( D\right) \right\Vert _{1-var,\left[ 0,T%
\right] }^{\frac{q}{2}}\right) &\leq &C_{d,q}\left( E\left( \left\Vert
Z\right\Vert _{4-var,\left[ 0,T\right] }^{q}\right) +E\left( \left\Vert
\left\langle Z\right\rangle \right\Vert _{1-var,\left[ 0,T\right] }^{\frac{q%
}{2}}\right) \right)  \label{inner bound on 1-var} \\
&\leq &C_{d,q}E\left( \left\vert \left\langle Z\right\rangle _{T}\right\vert
^{\frac{q}{2}}\right) <\infty \text{.}  \notag
\end{eqnarray}%
On the other hand, since $\lambda \left( D\right) $ is a discrete
martingale, by using BDG inequality and H\"{o}lder inequality, we get%
\begin{eqnarray}
&&E\left( \left\Vert \lambda \!\left( D\right) \right\Vert _{\infty -var, 
\left[ 0,T\right] }^{\frac{q}{2}}\right)  \label{inner estimation of lambdaD}
\\
&\leq &C_{d,q}E\left( \left\Vert \left\langle \lambda \!\left( D\right)
\right\rangle \right\Vert _{1-var,\left[ 0,T\right] }^{\frac{q}{4}}\right) 
\notag \\
&\leq &C_{d,q}\left( E\left( \left( \sum_{k=0}^{n-1}\left\vert
Z_{t_{k+1}}-Z_{t_{k}}\right\vert ^{4}\right) ^{\frac{q}{4}}\right) +E\left(
\left( \sum_{k=0}^{n-1}\left\vert \left\langle Z\right\rangle
_{t_{k},t_{k+1}}\right\vert ^{2}\right) ^{\frac{q}{4}}\right) \right)  \notag
\\
&\leq &C_{d,q}\left( E\left( \sup_{\left\vert t-s\right\vert \leq \left\vert
D\right\vert }\left\vert Z_{t}-Z_{s}\right\vert ^{q}\right) ^{\frac{1}{4}%
}E\left( \left\Vert Z\right\Vert _{3-var,\left[ 0,T\right] }^{q}\right) ^{%
\frac{3}{4}}+E\left( \sup_{\left\vert t-s\right\vert \leq \left\vert
D\right\vert }\left\vert \left\langle Z\right\rangle _{s,t}\right\vert ^{%
\frac{q}{2}}\right) ^{\frac{1}{2}}E\left( \left\Vert \left\langle
Z\right\rangle \right\Vert _{1-var,\left[ 0,T\right] }^{\frac{q}{2}}\right)
^{\frac{1}{2}}\right) \text{.}  \notag
\end{eqnarray}%
Then, by using $\left( \ref{BDG and Doob for local}\right) $ on page \pageref%
{BDG and Doob for local}, we get%
\begin{equation*}
E\left( \left\Vert Z\right\Vert _{3-var,\left[ 0,T\right] }^{q}\right) \leq
C_{d,q}E\left( \left\vert \left\langle Z\right\rangle _{T}\right\vert ^{%
\frac{q}{2}}\right) <\infty \text{.}
\end{equation*}%
On the other hand,%
\begin{equation*}
E\left( \left\Vert \left\langle Z\right\rangle \right\Vert _{1-var,\left[ 0,T%
\right] }^{\frac{q}{2}}\right) \leq C_{d,q}E\left( \left\vert \left\langle
Z\right\rangle _{T}\right\vert ^{\frac{q}{2}}\right) <\infty \text{.}
\end{equation*}%
Thus, we continue with $\left( \ref{inner estimation of lambdaD}\right) $,
and by dominated convergence theorem, we get%
\begin{equation}
\lim_{\left\vert D\right\vert \rightarrow 0}E\left( \left\Vert \lambda
\!\left( D\right) \right\Vert _{\infty -var,\left[ 0,T\right] }^{\frac{q}{2}%
}\right) =0\text{.}  \label{inner convergence in uniform norm}
\end{equation}%
Combining $\left( \ref{inner bound on 1-var}\right) $ and $\left( \ref{inner
convergence in uniform norm}\right) $, by interpolating between $1$%
-variation and uniform norm, we get that, for any $q\geq 1$ and $p>2$,%
\begin{equation*}
\lim_{\left\vert D\right\vert \rightarrow 0}E\left( \left\Vert \lambda
\!\left( D\right) \right\Vert _{\frac{p}{2}-var,\left[ 0,T\right] }^{\frac{q%
}{2}}\right) =0\text{.}
\end{equation*}

For the two terms left in $\left( \ref{inner separate estimate}\right) $, we
take $\left\langle Z\right\rangle ^{D}$ as an example. The estimation of $%
\left\langle Z\right\rangle $ is similar. Based on the definition of $%
\left\langle Z\right\rangle ^{D}$, we have%
\begin{equation*}
\sum_{k=0}^{n-1}\left\Vert \left\langle Z\right\rangle ^{D}\right\Vert _{%
\frac{p}{2}-var,\left[ t_{k},t_{k+1}\right] }^{\frac{p}{2}}\leq
\sum_{k=0}^{n-1}\left\vert Z_{t_{k+1}}-Z_{t_{k}}\right\vert ^{p}\leq \left(
\sup_{\left\vert t-s\right\vert \leq \left\vert D\right\vert }\left\vert
Z_{t}-Z_{s}\right\vert ^{\frac{p-2}{2}}\right) \left\Vert Z\right\Vert _{%
\frac{p+2}{2}-var,\left[ 0,T\right] }^{\frac{p+2}{2}}\text{.}
\end{equation*}%
Thus, by using H\"{o}lder inequality,%
\begin{equation}
E\left( \left( \sum_{k=0}^{n-1}\left\Vert \left\langle Z\right\rangle
^{D}\right\Vert _{\frac{p}{2}-var,\left[ t_{k},t_{k+1}\right] }^{\frac{p}{2}%
}\right) ^{\frac{q}{p}}\right) \leq E\left( \sup_{\left\vert t-s\right\vert
\leq \left\vert D\right\vert }\left\vert Z_{t}-Z_{s}\right\vert ^{q}\right)
^{\frac{p-2}{2p}}E\left( \left\Vert Z\right\Vert _{\frac{p+2}{2}-var,\left[
0,T\right] }^{q}\right) ^{\frac{p+2}{2p}}\text{.}
\label{inner estimation of p/2-var 1}
\end{equation}%
Thus, by using $\left( \ref{BDG and Doob for local}\right) $ (on p\pageref%
{BDG and Doob for local}) we have, for any $p>2$, 
\begin{equation*}
E\left( \left\Vert Z\right\Vert _{\frac{p+2}{2}-var,\left[ 0,T\right]
}^{q}\right) \leq C_{p,q,d}E\left( \left\vert \left\langle Z\right\rangle
_{T}\right\vert ^{\frac{q}{2}}\right) <\infty \text{.}
\end{equation*}%
Then, based on dominated convergence theorem, we have%
\begin{equation*}
\lim_{\left\vert D\right\vert \rightarrow 0}E\left( \left(
\sum_{k=0}^{n-1}\left\Vert \left\langle Z\right\rangle ^{D}\right\Vert _{%
\frac{p}{2}-var,\left[ t_{k},t_{k+1}\right] }^{\frac{p}{2}}\right) ^{\frac{q%
}{p}}\right) =0\text{.}
\end{equation*}
\end{proof}

\subsection{Rough path perturbed by martingale}

\subsubsection{Rough Path Underlying}

\begin{proof}[Proof of Theorem \protect\ref{Theorem general vector field for
Ito RDE}]
\label{Proof of Theorem general vector field for Ito RDE}Define $\omega
_{i}:\left\{ \left( s,t\right) |0\leq s\leq t\leq T\right\} \rightarrow 
\overline{%
\mathbb{R}
^{+}}$, $i=1,2$, as, for any $0\leq s\leq t\leq T$,%
\begin{eqnarray*}
\mathcal{\omega }_{1}\left( s,t\right) &:&=\left\Vert \gamma \right\Vert
_{p-var,\left[ s,t\right] }^{p}+\left\Vert \left\langle M\right\rangle
\right\Vert _{1-var,\left[ s,t\right] }^{\frac{p}{2}}\text{,} \\
\mathcal{\omega }_{2}\left( s,t\right) &:&=\left\Vert \gamma +S_{2}\left(
M\right) \right\Vert _{p-var,\left[ s,t\right] }^{p}\text{.}
\end{eqnarray*}%
Then $\mathcal{\omega }_{1}$ is deterministic and $\mathcal{\omega }%
_{1}\left( 0,T\right) <\infty $. Based on our assumption $\left( \ref%
{condition of integrability}\right) $ on p\pageref{condition of
integrability}, for some integer $n\geq 2$,%
\begin{equation}
E\left( \mathcal{\omega }_{2}\left( 0,T\right) ^{n}\right) <\infty \text{.}
\label{inner assumption on integrability of omega1}
\end{equation}%
Denote $\pi _{f}\left( s,\mathbf{\eta },\mathcal{I}_{2}\left( \gamma
,M\right) \right) $ as the solution to the rough differential equation: 
\begin{equation*}
dY=f\left( \pi _{1}\left( Y\right) \right) d\mathcal{I}_{2}\left( \gamma
,M\right) \text{, \ }y_{s}=\mathbf{\eta }\in T^{\left( n\right) }\left( 
\mathbb{R}
^{e}\right) \text{.}
\end{equation*}

For the selected integer $n\geq 2$ and finite partition $D=\left\{
t_{j}\right\} $ of $\left[ 0,T\right] $, denote $y^{n,D}:=y^{n,D}\left(
\gamma ,M\right) $ (defined at $\left( \ref{Definition of y_n^D}\right) $ on
p\pageref{Definition of y_n^D}) and denote $y_{s,t}^{n,D}:=\left(
y_{s}^{n,D}\right) ^{-1}\otimes y_{t}^{n,D}$ for $0\leq s\leq t\leq T$.
Recall $\left\{ y^{i,j}\right\} _{i=1,2}$ defined at $\left( \ref{Definition
of y^(i,k)}\right) $ on p\pageref{Definition of y^(i,k)}:%
\begin{eqnarray*}
dy_{u}^{1,j} &=&f\left( \pi _{1}\left( y_{u}^{1,j}\right) \right) d\gamma
_{u}\text{, \ }y_{t_{j}}^{1,j}=y^{n,D}\left( \gamma ,M\right) _{t_{j}}\in
T^{\left( n\right) }\left( 
\mathbb{R}
^{e}\right) \text{,} \\
dy_{u}^{2,j} &=&f\left( \pi _{1}\left( y_{u}^{2,j}\right) \right) d\left(
\gamma +S_{2}\left( M\right) \right) _{u}\text{, \ }y_{t_{j}}^{2,j}=y^{n,D}%
\left( \gamma ,M\right) _{t_{j}}\in T^{\left( n\right) }\left( 
\mathbb{R}
^{e}\right) \text{,}
\end{eqnarray*}%
Then, based on the definition of $y^{n,D}$, we have%
\begin{equation}
y_{t_{j},t_{j+1}}^{n,D}=y_{t_{j},t_{j+1}}^{1,j}\otimes E\left(
y_{t_{j},t_{j+1}}^{2,j}\right) ^{-1}\otimes y_{t_{j},t_{j+1}}^{1,j}\text{, \ 
}j\geq 0\text{.}  \label{inner definition of ynD on small interval}
\end{equation}%
Since $f$ is $Lip\left( \beta \right) $ for $\beta >p\geq 2$, $f$ is $%
Lip\left( 2\right) $. Based on Euler estimate of solution of RDE ($\left( %
\ref{taylor estimate of the first level of solution of RDE}\right) $ in
Theorem \ref{Theorem estimate of the solution of RDE}), we have, on any $%
\left[ t_{j},t_{j+1}\right] $,%
\begin{eqnarray*}
&&\left\vert \pi _{1}\left( y_{t_{j},t_{j+1}}^{n,D}\right) -\pi _{1}\left(
\pi _{f}\left( t_{j},y_{t_{j}}^{n,D},\mathcal{I}_{2}\left( \gamma ,M\right)
\right) _{t_{j},t_{j+1}}\right) \right\vert \\
&\leq &C_{p,f}\left( E\left( \mathcal{\omega }_{2}\left(
t_{j},t_{j+1}\right) ^{\frac{3}{p}}\right) +\mathcal{\omega }_{1}\left(
t_{j},t_{j+1}\right) ^{\frac{3}{p}}\right) \\
&&+\left\vert \left( Dff\right) \left( \pi _{1}\left( Y_{t_{j}}\right)
\right) \left( E\left( \int_{t_{j}}^{t_{j+1}}\left( M_{u}-M_{t_{j}}\right)
\otimes \circ dM_{u}\right) -\frac{1}{2}\left\langle M\right\rangle
_{t_{j},t_{j+1}}\right) \right\vert \text{.}
\end{eqnarray*}%
Since $M$ is in $L^{2}$,%
\begin{equation*}
E\left( \int_{t_{j}}^{t_{j+1}}\left( M_{u}-M_{t_{j}}\right) \otimes \circ
dM_{u}-\frac{1}{2}\left\langle M\right\rangle _{t_{j},t_{j+1}}\right)
=E\left( \int_{t_{j}}^{t_{j+1}}\left( M_{u}-M_{t_{j}}\right) \otimes
dM_{u}\right) =0\text{,}
\end{equation*}%
and we have ($M=\int \phi dB$ with $\phi $ a fixed path taking value in $%
d\times d$\ matrices) 
\begin{equation*}
E\left( \int_{t_{j}}^{t_{j+1}}\left( M_{u}-M_{t_{j}}\right) \otimes \circ
dM_{u}\right) =\frac{1}{2}E\left( \left\langle M\right\rangle
_{t_{j},t_{j+1}}\right) =\frac{1}{2}\left\langle M\right\rangle
_{t_{j},t_{j+1}}\text{.}
\end{equation*}%
Thus, for any $t_{j}\in D$,%
\begin{eqnarray}
&&\left\vert \pi _{1}\left( y_{t_{j},t_{j+1}}^{n,D}\right) -\pi _{1}\left(
\pi _{f}\left( t_{j},y_{t_{j}}^{n,D},\mathcal{I}_{2}\left( \gamma ,M\right)
\right) _{t_{j},t_{j+1}}\right) \right\vert
\label{inner Euler scheme for level 1} \\
&\leq &C_{p,f}\left( E\left( \mathcal{\omega }_{2}\left(
t_{j},t_{j+1}\right) ^{\frac{3}{p}}\right) +\mathcal{\omega }_{1}\left(
t_{j},t_{j+1}\right) ^{\frac{3}{p}}\right) \text{.}  \notag
\end{eqnarray}%
For the second level, based on $\left( \ref{inner definition of ynD on small
interval}\right) $, we have 
\begin{eqnarray}
\pi _{2}\left( y_{t_{j},t_{j+1}}^{n,D}\right) &=&\pi _{2}\left(
y_{t_{j},t_{j+1}}^{1,j}\otimes E\left( y_{t_{j},t_{j+1}}^{2,j}\right)
^{-1}\otimes y_{t_{j},t_{j+1}}^{1,j}\right)  \label{inner expression level 2}
\\
&=&2\pi _{2}\left( y_{t_{j},t_{j+1}}^{1,j}\right) -\pi _{2}\left( E\left(
y_{t_{j},t_{j+1}}^{2,j}\right) \right) +\left( \pi _{1}\left(
y_{t_{j},t_{j+1}}^{1,j}\right) -\pi _{1}\left( E\left(
y_{t_{j},t_{j+1}}^{2,j}\right) \right) \right) ^{\otimes 2}\text{.}  \notag
\end{eqnarray}%
Then, by using $\left( \ref{inner expression level 2}\right) $, combined
with $\left( \ref{taylor estimate of the second level of the solution of RDE}%
\right) $ in Theorem \ref{Theorem estimate of the solution of RDE}, we get,
(denote $\xi _{j}:=\pi _{1}\left( y^{D}\left( \gamma ,M\right)
_{t_{j}}\right) $)%
\begin{eqnarray}
&&\left\vert \pi _{2}\left( y_{t_{j},t_{j+1}}^{n,D}\right) -f\left( \xi
_{j}\right) \otimes f\left( \xi _{j}\right) \left( 2\pi _{2}\left( \gamma
_{t_{j},t_{j+1}}\right) -\left( \pi _{2}\left( \gamma
_{t_{j},t_{j+1}}\right) +\frac{1}{2}\left\langle M\right\rangle
_{t_{j},t_{j+1}}\right) \right) \right\vert
\label{inner estimate of second level of ito increment} \\
&\leq &C_{p,f}\left( E\left( \mathcal{\omega }_{2}\left(
t_{j},t_{j+1}\right) ^{\frac{3}{p}}\vee \mathcal{\omega }_{2}\left(
t_{j},t_{j+1}\right) ^{2}\right) +\mathcal{\omega }_{1}\left(
t_{j},t_{j+1}\right) ^{\frac{3}{p}}\vee \mathcal{\omega }_{1}\left(
t_{j},t_{j+1}\right) ^{2}\right)  \notag \\
&&+\left\vert \frac{1}{2}\left( Dff\right) \left( \xi _{j}\right)
\left\langle M\right\rangle _{t_{j},t_{j+1}}\right\vert ^{2}+C_{p,f}\left(
E\left( \mathcal{\omega }_{2}\left( t_{j},t_{j+1}\right) ^{\frac{3}{p}%
}\right) +\mathcal{\omega }_{1}\left( t_{j},t_{j+1}\right) ^{\frac{3}{p}%
}\right) ^{2}  \notag \\
&&+C_{p,f}\left( E\left( \mathcal{\omega }_{2}\left( t_{j},t_{j+1}\right) ^{%
\frac{3}{p}}\right) +\mathcal{\omega }_{1}\left( t_{j},t_{j+1}\right) ^{%
\frac{3}{p}}\right)  \notag \\
&&\times \left( E\left( \mathcal{\omega }_{2}\left( t_{j},t_{j+1}\right) ^{%
\frac{1}{p}}\vee \mathcal{\omega }_{2}\left( t_{j},t_{j+1}\right) \right) +%
\mathcal{\omega }_{1}\left( t_{j},t_{j+1}\right) ^{\frac{1}{p}}\vee \mathcal{%
\omega }_{1}\left( t_{j},t_{j+1}\right) \right) \text{.}  \notag
\end{eqnarray}%
On the other hand, again based on $\left( \ref{taylor estimate of the second
level of the solution of RDE}\right) $ in Theorem \ref{Theorem estimate of
the solution of RDE}, we have, 
\begin{eqnarray}
&&\left\vert \pi _{2}\left( \pi _{f}\left( t_{j},y_{t_{j}}^{n,D},\mathcal{I}%
_{2}\left( \gamma ,M\right) \right) _{t_{j},t_{j+1}}\right) -f\left( \xi
_{j}\right) \otimes f\left( \xi _{j}\right) \left( \pi _{2}\left( \gamma
_{t_{j},t_{j+1}}\right) -\frac{1}{2}\left\langle M\right\rangle
_{t_{j},t_{j+1}}\right) \right\vert
\label{inner estimate of second level of Ito solution} \\
&\leq &C_{p,f}\text{ }\mathcal{\omega }_{1}\left( t_{j},t_{j+1}\right) ^{%
\frac{3}{p}}\vee \mathcal{\omega }_{1}\left( t_{j},t_{j+1}\right) ^{2}\text{.%
}  \notag
\end{eqnarray}%
Therefore, combining $\left( \ref{inner estimate of second level of ito
increment}\right) $ and $\left( \ref{inner estimate of second level of Ito
solution}\right) $, we get,%
\begin{eqnarray}
&&\left\vert \pi _{2}\left( y_{t_{j},t_{j+1}}^{n,D}\right) -\pi _{2}\left(
\pi _{f}\left( t_{j},y_{t_{j}}^{n,D},\mathcal{I}_{2}\left( \gamma ,M\right)
\right) _{t_{j},t_{j+1}}\right) \right\vert
\label{inner Euler scheme for level 2} \\
&\leq &C\left( p,f,E\left( \mathcal{\omega }_{2}\left( 0,T\right)
^{2}\right) ,\mathcal{\omega }_{1}\left( 0,T\right) \right)  \notag \\
&&\times \left( E\left( \mathcal{\omega }_{2}\left( t_{j},t_{j+1}\right) ^{%
\frac{3}{p}}\vee \mathcal{\omega }_{2}\left( t_{j},t_{j+1}\right)
^{2}\right) +\mathcal{\omega }_{1}\left( t_{j},t_{j+1}\right) ^{\frac{3}{p}%
}\vee \mathcal{\omega }_{1}\left( t_{j},t_{j+1}\right) ^{2}\right) \text{.} 
\notag
\end{eqnarray}%
For the higher levels (i.e. $k\geq 3$), by using $\left( \ref{estimate of
solution of RDE}\right) $ in Theorem \ref{Theorem estimate of the solution
of RDE} and Young's inequality, we have%
\begin{eqnarray}
&&\left\vert \pi _{k}\left( y_{t_{j},t_{j+1}}^{n,D}\right) -\pi _{k}\left(
\pi _{f}\left( t_{j},y_{t_{j}}^{n,D},\mathcal{I}_{2}\left( \gamma ,M\right)
\right) _{t_{j},t_{j+1}}\right) \right\vert
\label{inner Euler scheme for higher levels} \\
&\leq &C\left( p,f,k\right) \left( E\left( \mathcal{\omega }_{2}\left(
t_{j},t_{j+1}\right) ^{\frac{k}{p}}\vee \mathcal{\omega }_{2}\left(
t_{j},t_{j+1}\right) ^{k}\right) +\mathcal{\omega }_{1}\left(
t_{j},t_{j+1}\right) ^{\frac{k}{p}}\vee \mathcal{\omega }_{1}\left(
t_{j},t_{j+1}\right) ^{k}\right) \text{.}  \notag
\end{eqnarray}%
Combine $\left( \ref{inner Euler scheme for level 1}\right) $, $\left( \ref%
{inner Euler scheme for level 2}\right) $ and $\left( \ref{inner Euler
scheme for higher levels}\right) $, if we define $\widetilde{\omega }%
_{k}:\left\{ \left( s,t\right) |0\leq s\leq t\leq T\right\} \rightarrow 
\overline{%
\mathbb{R}
^{+}}$ as%
\begin{equation}
\widetilde{\omega }_{k}\left( s,t\right) :=\left\{ 
\begin{array}{cc}
E\left( \mathcal{\omega }_{2}\left( s,t\right) ^{\frac{3}{p}}\right) +%
\mathcal{\omega }_{1}\left( s,t\right) ^{\frac{3}{p}}, & k=1 \\ 
E\left( \mathcal{\omega }_{2}\left( s,t\right) ^{\frac{3}{p}}\vee \mathcal{%
\omega }_{2}\left( s,t\right) ^{2}\right) +\mathcal{\omega }_{1}\left(
s,t\right) ^{\frac{3}{p}}\vee \mathcal{\omega }_{1}\left( s,t\right) ^{2}, & 
k=2 \\ 
E\left( \mathcal{\omega }_{2}\left( s,t\right) ^{\frac{k}{p}}\vee \mathcal{%
\omega }_{2}\left( s,t\right) ^{k}\right) +\mathcal{\omega }_{1}\left(
s,t\right) ^{\frac{k}{p}}\vee \mathcal{\omega }_{1}\left( s,t\right) ^{k}, & 
k\geq 3%
\end{array}%
\right. \text{,}  \label{Definition of omega tilde}
\end{equation}%
then%
\begin{eqnarray}
&&\left\vert \pi _{k}\left( y_{t_{j},t_{j+1}}^{n,D}\right) -\pi _{k}\left(
\pi _{f}\left( t_{j},y_{t_{j}}^{n,D},\mathcal{I}_{2}\left( \gamma ,M\right)
\right) _{t_{j},t_{j+1}}\right) \right\vert
\label{inner estimation of error on small interval} \\
&\leq &C\left( p,f,k,E\left( \mathcal{\omega }_{2}\left( 0,T\right)
^{2}\right) ,\mathcal{\omega }_{1}\left( 0,T\right) \right) \widetilde{%
\omega }_{k}\left( t_{j},t_{j+1}\right) \text{, }\forall j\geq 0\text{, }%
k=1,\dots ,n\text{.}  \notag
\end{eqnarray}%
Based on our assumption $\left( \ref{inner assumption on integrability of
omega1}\right) $ and that $n\geq 2$, we have 
\begin{equation}
\lim_{\left\vert D\right\vert \rightarrow 0}\sum_{t_{j}\in D}\widetilde{%
\omega }_{k}\left( t_{j},t_{j+1}\right) =0\text{, \ }k=1,\dots ,n\text{.}
\label{inner omega tilde tends to zero}
\end{equation}

Since $f$ is $Lip\left( \beta \right) $ for $\beta >p$, denote $Y$ as the
unique solution to the rough differential equation%
\begin{equation*}
dY=f\left( \pi _{1}\left( Y\right) \right) d\mathcal{I}_{2}\left( \gamma
,M\right) \text{, \ }Y_{0}=\mathbf{\xi }\in T^{\left( n\right) }\left( 
\mathbb{R}
^{e}\right) .
\end{equation*}%
We want to prove%
\begin{equation}
\lim_{\left\vert D\right\vert \rightarrow 0}\max_{1\leq k\leq n}\sup_{0\leq
t\leq T}\left\vert \pi _{k}\left( y_{t}^{n,D}\right) -\pi _{k}\left(
Y_{t}\right) \right\vert =0\text{.}  \label{inner goal of proof}
\end{equation}%
It is clear that 
\begin{equation*}
\pi _{0}\left( y_{t}^{n,D}\right) =\pi _{0}\left( Y_{t}\right) \equiv 1\text{%
,}
\end{equation*}%
so $\left( \ref{inner goal of proof}\right) $ holds trivially at level $0$.
Then we use mathematical induction. For integer $k=1,2,\dots ,n$, suppose $%
\left( \ref{inner goal of proof}\right) $ holds for level $l=0,1,\dots ,k-1$%
, we want to prove $\left( \ref{inner goal of proof}\right) $ at level $k$.
Based on our inductive hypothesis, we have%
\begin{equation}
\sup_{D\subset \left[ 0,T\right] }\max_{0\leq l\leq k-1}\sup_{0\leq t\leq
T}\left\vert \pi _{l}\left( y_{t}^{n,D}\right) \right\vert <\infty \text{.}
\label{inner bound}
\end{equation}%
For $t_{j}\in D$, when $j=0$, $y_{0}^{n,D}=Y_{0}=\mathbf{\xi }$. When $j=1$,
based on $\left( \ref{inner estimation of error on small interval}\right) $,
we have%
\begin{eqnarray*}
\left\vert \pi _{k}\left( y_{t_{1}}^{n,D}-Y_{t_{1}}\right) \right\vert
&=&\sum_{l=0}^{k-1}\left\vert \pi _{l}\left( \mathbf{\xi }\right)
\right\vert \left\vert \pi _{k-l}\left( y_{0,t_{1}}^{D}-Y_{0,t_{1}}\right)
\right\vert \\
&\leq &C\left( p,f,k,E\left( \mathcal{\omega }_{2}\left( 0,T\right)
^{2}\right) ,\mathcal{\omega }_{1}\left( 0,T\right) \right) \max_{0\leq
l\leq k-1}\left\vert \pi _{l}\left( \mathbf{\xi }\right) \right\vert
\sum_{l=1}^{k}\widetilde{\omega }_{l}\left( 0,t_{1}\right) \text{.}
\end{eqnarray*}%
When $j\geq 2$, we have%
\begin{eqnarray}
&&\left\vert \pi _{k}\left( y_{t_{j}}^{n,D}-Y_{t_{j}}\right) \right\vert
\label{inner uniform estimation} \\
&=&\sum_{i=0}^{j-1}\left\vert \pi _{k}\left( \pi _{f}\left(
t_{i+1},y_{t_{i+1}}^{n,D},\mathcal{I}_{2}\left( \gamma ,M\right) \right)
_{t_{j}}-\pi _{f}\left( t_{i},y_{t_{i}}^{n,D},\mathcal{I}_{2}\left( \gamma
,M\right) \right) _{t_{j}}\right) \right\vert  \notag \\
&\leq &\sum_{i=0}^{j-2}\left\vert \pi _{k}\left( \pi _{f}\left(
t_{i+1},y_{t_{i+1}}^{n,D},\mathcal{I}_{2}\left( \gamma ,M\right) \right)
_{t_{j}}-\pi _{f}\left( t_{i+1},\pi _{f}\left( t_{i},y_{t_{i}}^{n,D},%
\mathcal{I}_{2}\left( \gamma ,M\right) \right) _{t_{i+1}},\mathcal{I}%
_{2}\left( \gamma ,M\right) \right) _{t_{j}}\right) \right\vert  \notag \\
&&+\left\vert \pi _{k}\left( y_{t_{j}}^{n,D}-\pi _{f}\left(
t_{j-1},y_{t_{j-1}}^{n,D},\mathcal{I}_{2}\left( \gamma ,M\right) \right)
_{t_{j}}\right) \right\vert .  \notag
\end{eqnarray}%
Then for each $i=0,1,\dots ,j-2$, 
\begin{eqnarray*}
&&\pi _{f}\left( t_{i+1},y_{t_{i+1}}^{n,D},\mathcal{I}_{2}\left( \gamma
,M\right) \right) _{t_{j}}-\pi _{f}\left( t_{i+1},\pi _{f}\left(
t_{i},y_{t_{i}}^{n,D},\mathcal{I}_{2}\left( \gamma ,M\right) \right)
_{t_{i+1}},\mathcal{I}_{2}\left( \gamma ,M\right) \right) _{t_{j}} \\
&=&y_{t_{i+1}}^{n,D}\otimes \pi _{f}\left( t_{i+1},y_{t_{i+1}}^{n,D},%
\mathcal{I}_{2}\left( \gamma ,M\right) \right) _{t_{i+1},t_{j}} \\
&&-\pi _{f}\left( t_{i},y_{t_{i}}^{n,D},\mathcal{I}_{2}\left( \gamma
,M\right) \right) _{t_{i+1}}\otimes \pi _{f}\left( t_{i+1},\pi _{f}\left(
t_{i},y_{t_{i}}^{n,D},\mathcal{I}_{2}\left( \gamma ,M\right) \right)
_{t_{i+1}},\mathcal{I}_{2}\left( \gamma ,M\right) \right) _{t_{i+1},t_{j}} \\
&=&y_{t_{i+1}}^{n,D}\otimes \left( \pi _{f}\left( t_{i+1},y_{t_{i+1}}^{n,D},%
\mathcal{I}_{2}\left( \gamma ,M\right) \right) _{t_{i+1},t_{j}}-\pi
_{f}\left( t_{i+1},\pi _{f}\left( t_{i},y_{t_{i}}^{n,D},\mathcal{I}%
_{2}\left( \gamma ,M\right) \right) _{t_{i+1}},\mathcal{I}_{2}\left( \gamma
,M\right) \right) _{t_{i+1},t_{j}}\right) \\
&&+y_{t_{i}}^{n,D}\otimes \left( y_{t_{i},t_{i+1}}^{n,D}-\pi _{f}\left(
t_{i},y_{t_{i}}^{n,D},\mathcal{I}_{2}\left( \gamma ,M\right) \right)
_{t_{i},t_{i+1}}\right) \otimes \pi _{f}\left( t_{i+1},\pi _{f}\left(
t_{i},y_{t_{i}}^{n,D},\mathcal{I}_{2}\left( \gamma ,M\right) \right)
_{t_{i+1}},\mathcal{I}_{2}\left( \gamma ,M\right) \right) _{t_{i+1},t_{j}}%
\text{.}
\end{eqnarray*}%
Then use $\left( \ref{inner bound}\right) $, Theorem \ref{Theorem continuity
in initial value} on p\pageref{Theorem continuity in initial value} and $%
\left( \ref{inner estimation of error on small interval}\right) $ ($%
\widetilde{\omega }_{1}$ defined at $\left( \ref{Definition of omega tilde}%
\right) $), we have%
\begin{eqnarray*}
&&\left\vert \pi _{k}\left( y_{t_{i+1}}^{n,D}\otimes \left( \pi _{f}\left(
t_{i+1},y_{t_{i+1}}^{n,D},\mathcal{I}_{2}\left( \gamma ,M\right) \right)
_{t_{i+1},t_{j}}-\pi _{f}\left( t_{i+1},\pi _{f}\left( t_{i},y_{t_{i}}^{n,D},%
\mathcal{I}_{2}\left( \gamma ,M\right) \right) _{t_{i+1}},\mathcal{I}%
_{2}\left( \gamma ,M\right) \right) _{t_{i+1},t_{j}}\right) \right)
\right\vert \\
&\leq &\left( \sup_{D\subset \left[ 0,T\right] }\max_{0\leq l\leq
k-1}\sup_{0\leq t\leq T}\left\vert \pi _{l}\left( y_{t}^{n,D}\right)
\right\vert \right) \\
&&\times \left( \sum_{l=1}^{k}\left\vert \pi _{l}\left( \pi _{f}\left(
t_{i+1},y_{t_{i+1}}^{n,D},\mathcal{I}_{2}\left( \gamma ,M\right) \right)
_{t_{i+1},t_{j}}-\pi _{f}\left( t_{i+1},\pi _{f}\left( t_{i},y_{t_{i}}^{n,D},%
\mathcal{I}_{2}\left( \gamma ,M\right) \right) _{t_{i+1}},\mathcal{I}%
_{2}\left( \gamma ,M\right) \right) _{t_{i+1},t_{j}}\right) \right\vert
\right) \\
&\leq &C\left( p,\beta ,f,k,\mathcal{\omega }_{1}\left( 0,T\right) \right)
\left( \sup_{D\subset \left[ 0,T\right] }\max_{0\leq l\leq k-1}\sup_{0\leq
t\leq T}\left\vert \pi _{l}\left( y_{t}^{n,D}\right) \right\vert \right)
\left\vert \pi _{1}\left( y_{t_{i},t_{i+1}}^{n,D}-\pi _{f}\left(
t_{i},y_{t_{i}}^{n,D},\mathcal{I}_{2}\left( \gamma ,M\right) \right)
_{t_{i},t_{i+1}}\right) \right\vert \\
&\leq &C\left( p,\beta ,f,k,E\left( \mathcal{\omega }_{2}\left( 0,T\right)
^{2}\right) ,\mathcal{\omega }_{1}\left( 0,T\right) \right) \left(
\sup_{D}\max_{0\leq l\leq k-1}\sup_{0\leq t\leq T}\left\vert \pi _{l}\left(
y_{t}^{n,D}\right) \right\vert \right) \widetilde{\omega }_{1}\left(
t_{i},t_{i+1}\right)
\end{eqnarray*}%
On the other hand, use $\left( \ref{inner bound}\right) $, $\left( \ref%
{estimate of solution of RDE}\right) $ in Theorem \ref{Theorem estimate of
the solution of RDE} on p\pageref{Theorem estimate of the solution of RDE}
and $\left( \ref{inner estimation of error on small interval}\right) $ ($%
\widetilde{\omega }_{l}$ defined at $\left( \ref{Definition of omega tilde}%
\right) $), we have 
\begin{eqnarray*}
&&\left\vert \pi _{k}\left( y_{t_{i}}^{n,D}\otimes \left(
y_{t_{i},t_{i+1}}^{n,D}-\pi _{f}\left( t_{i},y_{t_{i}}^{n,D},\mathcal{I}%
_{2}\left( \gamma ,M\right) \right) _{t_{i},t_{i+1}}\right) \otimes \pi
_{f}\left( t_{i+1},\pi _{f}\left( t_{i},y_{t_{i}}^{n,D},\mathcal{I}%
_{2}\left( \gamma ,M\right) \right) _{t_{i+1}},\mathcal{I}_{2}\left( \gamma
,M\right) \right) _{t_{i+1},t_{j}}\right) \right\vert \\
&\leq &C\left( p,\beta ,f,k,E\left( \mathcal{\omega }_{2}\left( 0,T\right)
^{2}\right) ,\mathcal{\omega }_{1}\left( 0,T\right) \right) \left(
\sup_{D\subset \left[ 0,T\right] }\max_{0\leq l\leq k-1}\sup_{0\leq t\leq
T}\left\vert \pi _{l}\left( y_{t}^{n,D}\right) \right\vert \right)
\sum_{l=1}^{k}\widetilde{\omega }_{l}\left( t_{i},t_{i+1}\right) \text{.}
\end{eqnarray*}%
Therefore, we have, for any $i=0,1,\dots ,j-2$,%
\begin{eqnarray*}
&&\left\vert \pi _{k}\left( \pi _{f}\left( t_{i+1},y_{t_{i+1}}^{n,D},%
\mathcal{I}_{2}\left( \gamma ,M\right) \right) _{t_{j}}-\pi _{f}\left(
t_{i+1},\pi _{f}\left( t_{i},y_{t_{i}}^{n,D},\mathcal{I}_{2}\left( \gamma
,M\right) \right) _{t_{i+1}},\mathcal{I}_{2}\left( \gamma ,M\right) \right)
_{t_{j}}\right) \right\vert \\
&\leq &C\left( p,\beta ,f,k,E\left( \mathcal{\omega }_{2}\left( 0,T\right)
^{2}\right) ,\mathcal{\omega }_{1}\left( 0,T\right) \right) \left(
\sup_{D\subset \left[ 0,T\right] }\max_{0\leq l\leq k-1}\sup_{0\leq t\leq
T}\left\vert \pi _{l}\left( y_{t}^{n,D}\right) \right\vert \right)
\sum_{l=1}^{k}\widetilde{\omega }_{l}\left( t_{i},t_{i+1}\right) \text{.}
\end{eqnarray*}%
As a result,%
\begin{eqnarray}
&&\sum_{i=0}^{j-2}\left\vert \pi _{k}\left( \pi _{f}\left(
t_{i+1},y_{t_{i+1}}^{n,D},\mathcal{I}_{2}\left( \gamma ,M\right) \right)
_{t_{j}}-\pi _{f}\left( t_{i},y_{t_{i}}^{n,D},\mathcal{I}_{2}\left( \gamma
,M\right) \right) _{t_{j}}\right) \right\vert
\label{inner estimation for the first j-1 intervals} \\
&\leq &C\left( p,\beta ,f,k,E\left( \mathcal{\omega }_{2}\left( 0,T\right)
^{2}\right) ,\mathcal{\omega }_{1}\left( 0,T\right) \right) \left(
\sup_{D\subset \left[ 0,T\right] }\max_{0\leq l\leq k-1}\sup_{0\leq t\leq
T}\left\vert \pi _{l}\left( y_{t}^{n,D}\right) \right\vert \right)
\sum_{i=0}^{j-2}\left( \sum_{l=1}^{k}\widetilde{\omega }_{l}\left(
t_{i},t_{i+1}\right) \right) \text{.}  \notag
\end{eqnarray}%
On the other hand, for the term left in $\left( \ref{inner uniform
estimation}\right) $,%
\begin{eqnarray}
&&\left\vert \pi _{k}\left( y_{t_{j}}^{n,D}-\pi _{f}\left(
t_{j-1},y_{t_{j-1}}^{n,D},\mathcal{I}_{2}\left( \gamma ,M\right) \right)
_{t_{j}}\right) \right\vert  \label{inner estimation for the last interval}
\\
&=&\left\vert \pi _{k}\left( y_{t_{j-1}}^{n,D}\otimes \left(
y_{t_{j-1},t_{j}}^{n,D}-\pi _{f}\left( t_{j-1},y_{t_{j-1}}^{n,D},\mathcal{I}%
_{2}\left( \gamma ,M\right) \right) _{t_{j-1},t_{j}}\right) \right)
\right\vert  \notag \\
&\leq &C\left( p,\beta ,f,k,E\left( \mathcal{\omega }_{2}\left( 0,T\right)
^{2}\right) ,\mathcal{\omega }_{1}\left( 0,T\right) \right) \left(
\sup_{D\subset \left[ 0,T\right] }\max_{0\leq l\leq k-1}\sup_{0\leq t\leq
T}\left\vert \pi _{l}\left( y_{t}^{n,D}\right) \right\vert \right)
\sum_{l=1}^{k}\widetilde{\omega }_{l}\left( t_{j-1},t_{j}\right) \text{.} 
\notag
\end{eqnarray}%
Therefore, combining $\left( \ref{inner uniform estimation}\right) $, $%
\left( \ref{inner estimation for the first j-1 intervals}\right) $ and $%
\left( \ref{inner estimation for the last interval}\right) $, we have 
\begin{eqnarray*}
\left\vert \pi _{k}\left( y_{t_{j}}^{n,D}-Y_{t_{j}}\right) \right\vert &\leq
&C\left( p,\beta ,f,k,E\left( \mathcal{\omega }_{2}\left( 0,T\right)
^{2}\right) ,\mathcal{\omega }_{1}\left( 0,T\right) \right) \\
&&\times \left( \sup_{D\subset \left[ 0,T\right] }\max_{0\leq l\leq
k-1}\sup_{0\leq t\leq T}\left\vert \pi _{l}\left( y_{t}^{n,D}\right)
\right\vert \right) \sum_{i=0}^{j-1}\left( \sum_{l=1}^{k}\widetilde{\omega }%
_{l}\left( t_{i},t_{i+1}\right) \right)
\end{eqnarray*}%
Then, based on $\left( \ref{inner omega tilde tends to zero}\right) $ and
the inductive assumption $\left( \ref{inner bound}\right) $, we have 
\begin{equation}
\lim_{\left\vert D\right\vert \rightarrow 0}\max_{t_{j}\in D}\left\vert \pi
_{k}\left( y_{t_{j}}^{n,D}\right) -\pi _{k}\left( Y_{t_{j}}\right)
\right\vert =0\text{.}  \label{inner estimates convergence in prob}
\end{equation}

Since $y^{n,D}$ is piecewise-constant, we have%
\begin{equation}
\sup_{0\leq t\leq T}\left\vert \pi _{k}\left( y_{t}^{n,D}\right) -\pi
_{k}\left( Y_{t}\right) \right\vert \leq \max_{t_{j}\in D}\left\vert \pi
_{k}\left( y_{t_{j}}^{n,D}\right) -\pi _{k}\left( Y_{t_{j}}\right)
\right\vert +\sup_{\left\vert t-s\right\vert \leq \left\vert D\right\vert
}\left\vert \pi _{k}\left( Y_{t}\right) -\pi _{k}\left( Y_{s}\right)
\right\vert \text{.}  \label{inner estimates convergence in prob2}
\end{equation}%
For interval $\left[ s,t\right] $ satisfying $\left\Vert Y\right\Vert
_{p-var,\left[ s,t\right] }\leq 1$, we have,%
\begin{equation*}
\left\vert \pi _{k}\left( Y_{t}\right) -\pi _{k}\left( Y_{s}\right)
\right\vert =\left\vert \sum_{j=1}^{k}\pi _{k-j}\left( Y_{s}\right) \otimes
\pi _{j}\left( Y_{s,t}\right) \right\vert \leq C\left( k,\sup_{t\in \left[
0,T\right] }\left\Vert Y_{t}\right\Vert \right) \left\Vert Y\right\Vert
_{p-var,\left[ s,t\right] }.
\end{equation*}%
Since $Y$ is continuous and $\left\Vert Y\right\Vert _{p-var,\left[ 0,T%
\right] }<\infty $, we have%
\begin{equation*}
\lim_{\left\vert D\right\vert \rightarrow 0}\sup_{\left\vert t-s\right\vert
\leq \left\vert D\right\vert }\left\vert \pi _{k}\left( Y_{t}\right) -\pi
_{k}\left( Y_{s}\right) \right\vert =0\text{.}
\end{equation*}%
Combined with $\left( \ref{inner estimates convergence in prob}\right) $ and 
$\left( \ref{inner estimates convergence in prob2}\right) $, we get%
\begin{equation*}
\lim_{\left\vert D\right\vert \rightarrow 0,D\subset \left[ 0,T\right]
}\sup_{0\leq t\leq T}\left\vert \pi _{k}\left( y_{t}^{n,D}\right) -\pi
_{k}\left( Y_{t}\right) \right\vert =0\text{.}
\end{equation*}
\end{proof}

\subsubsection{Martingale underlying}

\begin{proof}[Proof of Corollary \protect\ref{Corollary for sample paths of
martingale}]
\label{Proof of Corollary of martingale underlying}$(Z,\widetilde{Z})$ is a $%
2d$-dimensional continuous martingale w.r.t. the filtration generated by $Z$
and $B$, so can be enhanced (by their Stratonovich integrals) to a $p$-rough
process for any $p>2$ (Theorem \ref{Theorem regulartiy of Stratonovich
signature of M} on p\pageref{Theorem regulartiy of Stratonovich signature of
M}). Suppose $Z$ is in $L^{2n+\epsilon }$ for some $\epsilon >0$ and integer 
$n\geq 1$. Using inequality $\left( \ref{BDG and Doob}\right) $ (on page %
\pageref{BDG and Doob}), we get (let $p:=2+n^{-1}\epsilon $) 
\begin{eqnarray*}
E\left( \left\Vert S_{2}\left( Z+\widetilde{Z}\right) \right\Vert _{p-var,%
\left[ 0,T\right] }^{np}\right) &\leq &C_{d,p,n}E\left( \left\Vert
\left\langle Z+\widetilde{Z}\right\rangle \right\Vert _{\infty -var,\left[
0,T\right] }^{2^{-1}np}\right) \\
&\leq &C_{d,p,n}E\left( \left\Vert \left\langle Z\right\rangle \right\Vert
_{\infty -var,\left[ 0,T\right] }^{2^{-1}np}\right) \\
&\leq &C_{d,p,n}E\left( \left\vert Z_{T}-Z_{0}\right\vert ^{np}\right)
=C_{d,p,n}E\left( \left\vert Z_{T}-Z_{0}\right\vert ^{2n+\epsilon }\right)
<\infty \text{.}
\end{eqnarray*}%
Thus, we have 
\begin{equation*}
E\left( \left\Vert S_{2}\left( Z+\widetilde{Z}\right) \right\Vert _{p-var,%
\left[ 0,T\right] }^{np}|Z\right) <\infty \text{ a.s..}
\end{equation*}%
On the other hand, fix a sample path of $Z$, we have that, the Stratonovich
integrals satisfy:%
\begin{equation*}
E\left( \int_{s}^{t}Z_{s,u}\otimes \circ d\widetilde{Z}_{u}+\int_{s}^{t}%
\widetilde{Z}_{s,u}\otimes \circ dZ_{u}|Z\right) =0\text{, }\forall 0\leq
s\leq t\leq T\text{, a.s..}
\end{equation*}%
Thus, based on Theorem \ref{Theorem general vector field for Ito RDE},
Corollary holds. (When $n=1$, it holds based on Remark \ref{Remark
convergence of first level}.)
\end{proof}

\section{Acknowledgement}

The research of both authors are supported by European Research Council
under the European Union's Seventh Framework Programme (FP7-IDEAS-ERC)/ ERC
grant agreement nr. 291244. The research of Terry Lyons is supported by
EPSRC grant EP/H000100/1. The authors are grateful for the support of
Oxford-Man Institute.


\begin{thebibliography}{99}
\bibitem{Bichteler} Bichteler, K., Stochastic integration and $L^{p}$-theory
of stochastic integration, \textit{Ann. Prob.}, \textbf{9}, 48-89, (1981).

\bibitem{CoutinQian} Coutin, L., Qian, Z., Stochastic analysis, rough path
analysis and fractional Brownian motion. \textit{Probab. Theory Related
Fields}, \textbf{122}(1),108-140, (2002).

\bibitem{Davie} Davie, A. M., Differential equations driven by rough paths:
an approach via discrete approximation. \textit{Appl. Math. Res. Express.
AMRX}, (2007).

\bibitem{Ethier and Kurtz} Ethier, S., Kurtz, T., \textit{Markov Processes,
Characterization and convergence}, John Wiley \& Sons, INC., (1985).

\bibitem{Follmer} F\"{o}llmer, H., Calcul d'It\^{o} sans probabilit\'{e}s, 
\textit{Seminaire de probabilites (Strasbourg)}, \textbf{15}, 143-150,
(1981).

\bibitem{FrizVicoirGuassian} Friz, P., Victoir, N., Differential equations
driven by Gaussian signals. \textit{Ann. Inst. H. Poincar\'{e} Probab.
Statist}., \textbf{46}(2), 369-413, (2010).

\bibitem{Peter Friz} Friz, P., Victoir, N., \textit{Multidimensional
Stochastic Processes as Rough Paths, Theory and Applications}, Cambridge
Univ. Press, (2010).

\bibitem{Gubinelli} Gubinelli, M., Ramification of rough paths, \textit{J.
Differential Equations}, \textbf{248}(2), 693-721, (2010).

\bibitem{Hairer and Kelly} Hairer, M., Kelly, D., Geometric versus
non-geometric rough paths, preprint, (2013).

\bibitem{HairerWeber} Harier, M., Weber, H., Rough Burgers-like equations
with multiplicative noise, \textit{Probab. Theory Related Fields}, \textbf{%
155}(1), 71-126, (2013).

\bibitem{Ito1} It\^{o}, K., Stochastic integral, \textit{Proc. Imp. Acad.
Tokyo, }\textbf{20}, 519-524, (1944).

\bibitem{Ito2} It\^{o}, K., On stochastic differential equations. Proc. 
\textit{Imp. Acad. Tokyo}, \textbf{22}, 32-35, (1946).

\bibitem{Karandikar} Karandikar, R. L., Pathwise solution of stochastic
differential equations, \textit{Sankhya A}, \textbf{43}, 121-132, (1981).

\bibitem{Lejay and Victoir} Lejay, A., Victoir, N., On $\left( p,q\right) $%
-rough paths, \textit{Journal of Differential Equations}, \textbf{225},
103-133, (2006).

\bibitem{Terry98} Lyons, T. J., Differential equations driven by rough
signals, \textit{Rev. Mat. Iberoamericana,} \textbf{14}(2), 215-310, (1998).

\bibitem{Terrysnotes} Lyons, T. J., Caruana, M., L\'{e}vy, T., Picard, J., 
\textit{Differential equation driven by rough paths}. Springer, (2007).

\bibitem{TerryQianpaper} Lyons, T. J., Qian, Z. M., Calculus for
multiplicative functionals, It\^{o}'s lemma and differential equations,%
\textit{\ It\^{o}'s Stochastic Calculus and Probability Theory}, Springer,
Tokyo, 233-250, (1996).

\bibitem{TerryQian} Lyons, T. J., Qian., Z., \textit{System control and
rough paths}, OxfordUniv. Press, (2002).

\bibitem{TerryStoica} Lyons, T. J., Stoica, L., The limits of stochastic
integrals of differential forms. \textit{Ann. Probab.}, \textbf{27}(1),
1-49, (1999).

\bibitem{TerryYang} Lyons, T. J., Yang, D., The partial sum process of
orthogonal expansion as geometric rough process with Fourier series as an
example---an improvement of Menshov-Rademacher theorem, (2011).

\bibitem{Revuz and Yor} Revuz, D., Yor, M., \textit{Continuous martingales
and Brownian motion}, 3rd ed, Springer, (1999).

\bibitem{Russo and Vallois} Russo, F., Vallois, P., Int\'{e}grales
progressive, r\'{e}trograde et sym\'{e}trique de processus nonadapt\'{e}s. 
\textit{C. R. Acad. Sci. Paris Ser. I Math.}, \textbf{312}(8), 615-618,
(1991).

\bibitem{Wong Zakai} Wong, E., Zakai, M., On the relation between ordinary
and stochastic differential equations, \textit{Int. J. Engineering Sci.}, 
\textbf{3}(2), 213-229, (1965).

\bibitem{Young L C} Young, L. C., An inequality of H\"{o}lder type,
connected with Stieltjes integration.\textit{\ Acta Math.} \textbf{67},
251-282, (1936).
\end{thebibliography}
\end{document}